\documentclass[a4paper,10pt]{article}

\usepackage{a4wide}
\usepackage[utf8]{inputenc}
\usepackage{amsfonts,amsmath,amsthm,amssymb}
\usepackage{graphicx, color}

\newtheorem{algorithm}{Algorithm}[section]

\newtheorem{lemma}{Lemma}[section]

\newtheorem{definition}{Definition}[section]
\newtheorem{theorem}{Theorem}[section]
\newtheorem{proposition}{Proposition}[section]

\newtheorem{remark}{Remark}[section]

\newtheorem{assumption}{Assumption}

\newcommand{\norm}[1]{\Vert #1 \Vert}
\newcommand{\mv}{\,\vert\,}
\newcommand{\R}{\mathbb{R}}
\newcommand{\N}{\mathbb{N}}

\newcommand{\Itl}[2]{\par\noindent\makebox[#1\parindent][l]{#2}\ignorespaces}

\newcommand{\setto}[1]{\mathop \to\limits^{#1}}

\newcommand{\Beta}{\theta}

\newcommand{\email}[1]{{\tt #1}}

\begin{document}

\title{An SQP method for mathematical programs with vanishing constraints with strong convergence properties}
\author{Mat\'u\v{s} Benko, Helmut Gfrerer\thanks{Institute of Computational Mathematics, Johannes Kepler University Linz,
              A-4040 Linz, Austria, \email{benko@numa.uni-linz.ac.at}, \email{helmut.gfrerer@jku.at}}}

\date{}
        
\maketitle

\begin{abstract}
  We propose an SQP algorithm for mathematical programs with vanishing constraints which solves at each iteration a quadratic
  program with linear vanishing constraints. The algorithm is based on the newly developed concept of $\mathcal Q$-stationarity \cite{BeGfr16b}.
  We demonstrate how $\mathcal Q_M$-stationary solutions of the quadratic program can be obtained.
  We show that all limit points of the sequence of iterates generated by the basic SQP method are at least M-stationary
  and by some extension of the method we also guarantee the stronger property of $\mathcal Q_M$-stationarity of the limit points.
  
  
\end{abstract}

\section{Introduction}
   
   Consider the following {\em mathematical program with vanishing constraints} (MPVC)
   \begin{equation} \label{eq : genproblem}
      \begin{array}{rll}
	\min\limits_{x \in \mathbb{R}^n} & f(x) & \\
	\textrm{subject to } & h_i(x) = 0 & i \in E,\\
	& g_i(x) \leq 0 & i \in I,\\
	& H_i(x) \geq 0, \, G_i(x) H_i(x) \leq 0 & i \in V,
      \end{array}
    \end{equation}
    with continuously differentiable functions $f$, $h_i, i \in E$, $g_i, i \in I$, $G_i, H_i, i \in V$ and finite index sets $E,I$ and $V$.
    
    Theoretically, MPVCs can be viewed as standard nonlinear optimization problems, but due to the vanishing
    constraints, many of the standard constraint qualifications of nonlinear programming are violated at any feasible point $\bar x$
    with $H_i(\bar x) = G_i(\bar x) = 0$ for some $i \in V$.
    On the other hand, by introducing slack variables, MPVCs may be reformulated as so-called
    mathematical programs with complementarity constraints (MPCCs), see \cite{Ho09}.
    However, this approach is also not satisfactory as it has turned out that MPCCs are in fact even more difficult to handle than MPVCs.
    This makes it necessary, both from a theoretical and numerical point of view, to consider special tailored algorithms
    for solving MPVCs. Recent numerical methods follow different directions.
    A smoothing-continuation method and a regularization approach for MPCCs are considered in \cite{FuPa99,Sch01}
    and a combination of these techniques, a smoothing-regularization approach for MPVCs is investigated in \cite{AchHoKa13}.
    In \cite{IzSo09,AchKaHo12} the relaxation method has been suggested in order to deal with the inherent difficulties of MPVCs.
    
    In this paper, we carry over a well known SQP method from nonlinear programming to MPVCs.
    We proceed in a similar manner as in \cite{BeGfr16a}, where an SQP method for MPCCs was introduced by Benko and Gfrerer.
    The main task of our method is to solve in each iteration step a quadratic program with linear vanishing constraints,
    a so-called auxiliary problem.
    Then we compute the next iterate by reducing a certain merit function along some polygonal line which is given by the solution
    procedure for the auxiliary problem. To solve the auxiliary problem we exploit the new concept of {\em $\mathcal Q_M$-stationarity}
    introduced in the recent paper by Benko and Gfrerer \cite{BeGfr16b}.
    $\mathcal Q_M$-stationarity is in general stronger than M-stationarity
    and it turns out to be very suitable for a numerical approach as it allows to handle
    the program with vanishing constraints without relying on enumeration techniques.
    Surprisingly, we compute at least a $\mathcal Q_M$-stationary solution of the auxiliary problem
    just by means of quadratic programming by solving appropriate convex subproblems.
    
    Next we study the convergence of the SQP method. We show that every limit point of the generated sequence is at least M-stationary.
    Moreover, we consider the extended version of our SQP method, where at each iterate a correction of the iterate is made
    to prevent the method from converging to undesired points.
    Consequently we show that under some additional assumptions all limit points are at least $\mathcal Q_M$-stationary.
    Numerical tests indicate that our method behaves very reliably.
    
    A short outline of this paper is as follows.
    In section 2 we recall the basic stationarity concepts for MPVCs as well as
    the recently developed concepts of $\mathcal Q$- and $\mathcal Q_M$-stationarity.
    In section 3 we describe an algorithm based on quadratic programming
    for solving the auxiliary problem occurring in every iteration of our SQP method.
    We prove the finiteness and summarize some other properties of this algorithm.
    In section 4 we propose the basic SQP method. We describe how the next iterate is computed by means of the solution of the auxiliary problem
    and we consider the convergence of the overall algorithm.
    In section 5 we consider the extended version of the overall algorithm and we discuss its convergence.
    Section 6 is a summary of numerical results we obtained by implementing our basic algorithm in MATLAB and by testing it on
    a subset of test problems considered in the thesis of Hoheisel \cite{Ho09}.
    
    In what follows we use the following notation. Given a set $M$ we denote by
    $\mathcal{P}(M):=\{ (M_1,M_2) \mv M_1 \cup M_2 = M, \, M_1 \cap M_2 = \emptyset \}$
    the collection of all partitions of $M$. Further, for a real number $a$ we use the notation $(a)^+:=\max(0,a)$, $(a)^-:=\min(0,a)$.
    For a vector $u= (u_1, u_2, \ldots, u_m)^T \in \R^m$ we define $\vert u \vert$, $(u)^+$, $(u)^-$ componentwise, i.e.
    $\vert u \vert := (\vert u_1 \vert, \vert u_2 \vert, \ldots, \vert u_m \vert)^T$, etc.
    Moreover, for $u \in \R^m$ and $1 \leq p \leq \infty$ we denote the $\ell_p$ norm of $u$ by $\norm{u}_p$
    and we use the notation $\norm{u} := \norm{u}_2$ for the standard $\ell_2$ norm.
    Finally, given a sequence $y_k \in \R^m$, a point $y \in \R^m$ and an infinite set $K \subset \N$ we write $y_k \setto{K} y$
    instead of $\lim_{k \to \infty, k \in K} y_k = y$.

\section{Stationary points for MPVCs}
  
  Given a point $\bar x$ feasible for \eqref{eq : genproblem} we define the following index sets
  \begin{eqnarray} \nonumber
      I^g(\bar x) & := & \{ i \in I \mv g_i(\bar x) = 0 \}, \\ \nonumber
      I^{0+}(\bar x) & := & \{ i \in V \mv H_i(\bar x) = 0 < G_i(\bar x) \}, \\ \label{eqn : IndexStes}
      I^{0-}(\bar x) & := & \{ i \in V \mv H_i(\bar x) = 0 > G_i(\bar x) \}, \\ \nonumber
      I^{+0}(\bar x) & := & \{ i \in V \mv H_i(\bar x) > 0 = G_i(\bar x) \}, \\ \nonumber
      I^{00}(\bar x) & := & \{ i \in V \mv H_i(\bar x) = 0 = G_i(\bar x) \}, \\ \nonumber
      I^{+-}(\bar x) & := & \{ i \in V \mv H_i(\bar x) > 0 < G_i(\bar x) \}.
  \end{eqnarray}

  In contrast to nonlinear programming there exist a lot of stationarity concepts for MPVCs.
  \begin{definition}
    Let $\bar x$ be feasible for \eqref{eq : genproblem}. Then $\bar x$ is called
    \begin{enumerate}
     \item {\em weakly stationary}, if there are multipliers $\lambda_i^g, i \in I$, $\lambda_i^h, i \in E$, $\lambda_i^G, \lambda_i^H, i \in V$
     such that
     \begin{equation} \label{eq : StatEq}
      \nabla f(\bar x)^T + \sum_{i \in E} \lambda_i^h \nabla h_i(\bar x)^T + \sum_{i \in I} \lambda_i^g \nabla g_i(\bar x)^T
      + \sum_{i \in V} \left( - \lambda_i^H \nabla H_i(\bar x)^T + \lambda_i^G \nabla G_i(\bar x)^T \right) = 0
     \end{equation}
     and
     \begin{equation} \label{eq : WeakStat}
      \begin{array}{rcl}
      \lambda_i^g g_i(\bar x) = 0, i \in I, & \lambda_i^H H_i(\bar x) = 0, i \in V, & \lambda_i^G G_i(\bar x) = 0, i \in V, \\
      \lambda_i^g \geq 0, i \in I, & \lambda_i^H \geq 0, i \in I^{0-}(\bar x), & \lambda_i^G \geq 0, i \in I^{00}(\bar x) \cup I^{+0}(\bar x).
     \end{array}
     \end{equation}
      \item {\em M-stationary}, if it is weakly stationary and
      \begin{equation} \label{eq : MStatCond}
	\lambda_i^H \lambda_i^G = 0, i \in I^{00}(\bar x).
      \end{equation}
      \item {\em $\mathcal{Q}$-stationary with respect to $(\beta^1,\beta^2)$}, where $(\beta^1,\beta^2)$ is a given partition of $I^{00}(\bar x)$,
      if there exist two multipliers $\overline\lambda=(\overline\lambda^h,\overline\lambda^g,\overline\lambda^H,\overline\lambda^G)$
      and $\underline\lambda=(\underline\lambda^h,\underline\lambda^g,\underline\lambda^H,\underline\lambda^G)$,
      both fulfilling \eqref{eq : StatEq} and \eqref{eq : WeakStat}, such that
      \begin{equation} \label{eq : QStatCond}
	\overline\lambda_i^G = 0, \ \underline\lambda_i^H, \underline\lambda_i^G \geq 0, \ i \in \beta^1; \quad 
	\overline\lambda_i^H, \overline\lambda_i^G \geq 0, \ \underline\lambda_i^G = 0, \ i \in \beta^2.
      \end{equation}
      \item {\em $\mathcal{Q}$-stationary}, if there is some partition $(\beta^1,\beta^2) \in \mathcal P(I^{00}(\bar x))$ such that $\bar x$ is
      $\mathcal{Q}$-stationary with respect to $(\beta^1,\beta^2)$.
      \item {\em $\mathcal{Q}_M$-stationary}, if it is $\mathcal{Q}$-stationary and at least one of the multipliers $\overline\lambda$ and
      $\underline\lambda$ fulfills M-stationarity condition \eqref{eq : MStatCond}.
      \item {\em S-stationary}, if it is weakly stationary and
      \[ \lambda_i^H \geq 0, \lambda_i^G = 0, i \in I^{00}(\bar x). \]
    \end{enumerate}
  \end{definition}
  The concepts of $\mathcal{Q}$-stationarity and $\mathcal{Q}_M$-stationarity
  were introduced in the recent paper by Benko and Gfrerer \cite{BeGfr16b}, whereas the other stationarity concepts
  are very common in the literature, see e.g. \cite{AchKa08,Ho09,IzSo09}.  
  The following implications hold:
  \begin{eqnarray*}
    & \textrm{S-stationarity} \Rightarrow \mathcal{Q}\textrm{-stationarity with respect to every }
    (\beta^1,\beta^2) \in \mathcal{P}(I^{00}(\bar x)) \Rightarrow & \\
    & \mathcal{Q}\textrm{-stationarity w.r.t. } (\emptyset,I^{00}(\bar x)) \Rightarrow
    \mathcal{Q}_M\textrm{-stationarity} \Rightarrow \textrm{M-stationarity} \Rightarrow \textrm{weak stationarity}.&
  \end{eqnarray*}
  The first implication follows from the fact that the multiplier corresponding to S-stationarity fulfills
  the requirements for both $\overline\lambda$ and $\underline\lambda$. The third implication holds because
  for $(\beta^1,\beta^2) = (\emptyset,I^{00}(\bar x))$ the multiplier $\underline\lambda$ fulfills
  \eqref{eq : MStatCond} since $\underline\lambda_i^G = 0$ for $i \in I^{00}(\bar x)$.
  
  Note that the S-stationarity conditions are nothing else than the Karush-Kuhn-Tucker conditions for the problem \eqref{eq : genproblem}.
  As we will demonstrate in the next theorems, a local minimizer is S-stationary only under some comparatively stronger constraint qualification,
  while it is $\mathcal{Q}_M$-stationary under very weak constraint qualifications. Before stating the theorems we recall some
  common definitions.
    
  Denoting
  \begin{eqnarray} \label{eqn : FPdef}
    F_i(x):=(-H_i(x),G_i(x))^T, i \in V, && P := \{(a,b) \in \R_- \times \R \mv ab \geq 0 \}, \\ \label{eqn : mathcFdef}
    \mathcal{F}(x) := (h(x)^T,g(x)^T,F(x)^T)^T, && D:= \{0\}^{\vert E \vert} \times \R_-^{\vert I \vert} \times P^{\vert V \vert},
  \end{eqnarray}
  we see that problem \eqref{eq : genproblem} can be rewritten as
  \[\min f(x) \quad \textrm{subject to} \quad x \in \Omega_V := \{x \in \R^n \mv \mathcal{F}(x) \in D \}.\]
  Recall that the {\em contingent} (also {\em tangent}) {\em cone} to a closed
  set $\Omega \subset \R^m$ at $u \in \Omega$ is defined by
  \[T_{\Omega}(u) := \{ d \in \R^m \mv \exists (d_k) \to d, \exists (\tau_k) \downarrow 0 : u + \tau_k d_k \in \Omega \, \forall k \}. \]
  The {\em linearized cone} to $\Omega_V$ at $\bar x \in \Omega_V$ is then defined as
  $T_{\Omega_V}^{\mathrm{lin}}(\bar x) := \{d \in \R^n \mv \nabla \mathcal{F}(\bar x) d \in T_{D}(\mathcal{F}(\bar x))\}$.
  
  Further recall that $\bar x \in \Omega_V$ is called {\em B-stationary} if
  \[\nabla f(\bar x) d \geq 0 \, \forall d \in T_{\Omega_V}(\bar x).\]
  Every local minimizer is known to be B-stationary.
    
  \begin{definition}
    Let $\bar x$ be feasible for \eqref{eq : genproblem}, i.e $\bar x \in \Omega_V$. We say that the {\em generalized Guignard constraint qualification}
    (GGCQ) holds at $\bar x$, if the polar cone of $T_{\Omega_V}(\bar x)$ equals the polar cone of $T_{\Omega_V}^{\mathrm{lin}}(\bar x)$.
  \end{definition}
  
  \begin{theorem}[{c.f. \cite[Theorem 8]{BeGfr16b}}]
    Assume that GGCQ is fulfilled at the point $\bar x \in \Omega_V$. If $\bar x$ is B-stationary,
    then $\bar x$ is $\mathcal{Q}$-stationary for \eqref{eq : genproblem} with respect to every partition
    $(\beta^1,\beta^2) \in \mathcal{P}(I^{00}(\bar x))$ and it is also $\mathcal{Q}_M$-stationary.
  \end{theorem}
  
  \begin{theorem}[{c.f. \cite[Theorem 8]{BeGfr16b}}]
    If $\bar x$ is Q-stationary with respect to a partition $(\beta^1,\beta^2) \in \mathcal{P}(I^{00}(\bar x))$,
    such that for every $j \in \beta^1$ there exists some $z^j$ fulfilling
    \begin{equation} \label{eq : MPVCMFCQ1}
    \begin{array}{l}
      \nabla h(\bar x) z^j = 0, \\
      \nabla g_i(\bar x) z^j = 0, i \in I^g(\bar x), \\
      \nabla G_i(\bar x) z^j = 0, i \in I^{+0}(\bar x), \\
      \nabla G_i(\bar x) z^j \left\{
	\begin{array}{lr}
	  \geq 0, & i \in \beta^1,\\
	  \leq 0, & i \in \beta^2,
        \end{array}
      \right. \\
      \nabla H_i(\bar x) z^j = 0, i \in I^{0-}(\bar x) \cup I^{00}(\bar x) \cup I^{0+}(\bar x) \setminus \{j\}, \\
      \nabla H_j(\bar x) z^j = -1
    \end{array}
  \end{equation}
  and there is some $\bar z$ such that
  \begin{equation} \label{eq : MPVCMFCQ2}
    \begin{array}{l}
      \nabla h(\bar x) \bar z = 0, \\
      \nabla g_i(\bar x) \bar z = 0, i \in I^g(\bar x), \\
      \nabla G_i(\bar x) \bar z = 0, i \in I^{+0}(\bar x), \\
      \nabla G_i(\bar x) \bar z \left\{
	\begin{array}{lr}
	  \geq 0, & i \in \beta^1,\\
	  \leq -1, & i \in \beta^2,
        \end{array}
      \right. \\
      \nabla H_i(\bar x) \bar z = 0, i \in I^{0-}(\bar x) \cup I^{00}(\bar x) \cup I^{0+}(\bar x),
    \end{array}
  \end{equation}
    then $\bar x$ is S-stationary and consequently also B-stationary.
  \end{theorem}
  
  Note that these two theorems together also imply that a local minimizer $\bar x \in \Omega_V$ is S-stationary provided GGCQ is fulfilled at $\bar x$
  and there exists a partition $(\beta^1,\beta^2) \in \mathcal{P}(I^{00}(\bar x))$,
  such that for every $j \in \beta^1$ there exists $z^j$ fulfilling \eqref{eq : MPVCMFCQ1} and $\bar z$ fulfilling \eqref{eq : MPVCMFCQ2}.
  
  Moreover, note that \eqref{eq : MPVCMFCQ1} and \eqref{eq : MPVCMFCQ2} are fulfilled for every partition $(\beta^1,\beta^2) \in \mathcal{P}(I^{00}(\bar x))$
  e.g. if the gradients of active constraints are linearly independent. On the other hand, in the special case of partition
  $(\emptyset,I^{00}(\bar x)) \in \mathcal{P}(I^{00}(\bar x))$, this conditions read as the requirement that the system
  \begin{equation*} \label{eq : MPVCMFCQ2special}
    \begin{array}{l}
      \nabla h(\bar x) \bar z = 0, \\
      \nabla g_i(\bar x) \bar z = 0, i \in I^g(\bar x), \\
      \nabla G_i(\bar x) \bar z = 0, i \in I^{+0}(\bar x), \\
      \nabla G_i(\bar x) \bar z \leq -1, i \in I^{00}(\bar x), \\
      \nabla H_i(\bar x) \bar z = 0, i \in I^{0-}(\bar x) \cup I^{00}(\bar x) \cup I^{0+}(\bar x)
    \end{array}
  \end{equation*}
  has a solution, which resembles the well-known Mangasarian-Fromovitz constraint qualification (MFCQ) of nonlinear programming
  and it seems to be a rather weak and possibly often fulfilled assumption.
  
  Finally, we recall the definitions of normal cones.
  The {\em regular normal cone} to a closed set $\Omega \subset \R^m$ at $u \in \Omega$ can be defined as the polar
  cone to the tangent cone by
  \[
    \widehat N_{\Omega}(u) := (T_{\Omega}(u))^{\circ} = \{ z \in \R^m \mv (z,d) \leq 0 \, \forall d \in T_{\Omega}(u)\}.
  \]
  The {\em limiting normal cone} to a closed set $\Omega \subset \R^m$ at $u \in \Omega$ is given by
  \begin{equation} \label{eq : LimitNCdef}
    N_{\Omega}(u) := \{ z \in \R^m \mv \exists u_k \to u, z_k \to z \textrm{ with } u_k \in \Omega, z_k \in \widehat N_{\Omega}(u_k) \, \forall k \}.
  \end{equation} 
  In case when $\Omega$ is a convex set, regular and limiting normal cone coincide with the classical normal cone of convex analysis, i.e.
  \begin{equation} \label{eq : convexNC}
    \widehat N_{\Omega}(u) = N_{\Omega}(u) = \{z \in \R^m \mv (z,u - v) \leq 0 \, \forall v \in \Omega\}.
  \end{equation}
  Well-known is also the following description of the limiting normal cone
  \begin{equation} \label{eq : propLimitNC}
    N_{\Omega}(u) := \{ z \in \R^m \mv \exists u_k \to u, z_k \to z \textrm{ with } u_k \in \Omega, z_k \in N_{\Omega}(u_k) \, \forall k \}.
  \end{equation}
  
  We conclude this section by the following characterization of M- and $\mathcal Q$-stationarity via limiting normal cone.
  Straightforward calculations yield that
  \begin{eqnarray*}
    & N_{P}(F_i(\bar x)) = \left\{
    \begin{array}{ll}
      \R_+ \times \{0\} & \textrm{if } i \in I^{0-}(\bar x), \\
      \R \times \{0\} \cup \{0\} \times \R_+ & \textrm{if } i \in I^{00}(\bar x), \\
      \R \times \{0\} & \textrm{if } i \in I^{0+}(\bar x), \\
      \{0\} \times \R_+ & \textrm{if } i \in I^{+0}(\bar x), \\
      \{0\} \times \{0\} & \textrm{if } i \in I^{+-}(\bar x),
    \end{array} \right.& \\
    & N_{P^1}(F_i(\bar x)) = \R \times \{0\} \quad \textrm{ if } i \in I^{0+}(\bar x) \cup I^{00}(\bar x) \cup I^{0-}(\bar x),& \\
    &N_{P^2}(F_i(\bar x)) = \left\{
    \begin{array}{ll}
      \R_+ \times \R_+ & \textrm{if } i \in I^{00}(\bar x), \\
      N_{P}(F_i(\bar x)) & \textrm{if } i \in I^{0-}(\bar x) \cup I^{+0}(\bar x) \cup I^{+-}(\bar x)
    \end{array} \right.&
  \end{eqnarray*}
  and hence the M-stationarity conditions \eqref{eq : WeakStat} and \eqref{eq : MStatCond} can be replaced by
  \begin{equation} \label{eq : MstatLNC}
  (\lambda^h,\lambda^g,\lambda^H,\lambda^G) \in N_{D}(\mathcal{F}(\bar x))
  = \R^{\vert E \vert} \times \{u \in \R_+^{\vert I \vert} \mv (u,g(\bar x)) = 0\} \times N_{P^{\vert V \vert}}(F(\bar x))
  \end{equation}
  and the $\mathcal Q$-stationarity conditions \eqref{eq : WeakStat} and \eqref{eq : QStatCond} can be replaced by
  \begin{eqnarray} \label{eqn : QstatLNC1}
  (\overline\lambda^h,\overline\lambda^g,\overline\lambda^H,\overline\lambda^G) & \in &
  \R^{\vert E \vert} \times \{u \in \R_+^{\vert I \vert} \mv (u,g(\bar x)) = 0\} \times \prod_{i \in V} \nu_i^{\beta^1,\beta^2}(\bar x), \\ \label{eqn : QstatLNC2}
  (\underline\lambda^h,\underline\lambda^g,\underline\lambda^H,\underline\lambda^G) & \in &
  \R^{\vert E \vert} \times \{u \in \R_+^{\vert I \vert} \mv (u,g(\bar x)) = 0\} \times \prod_{i \in V} \nu_i^{\beta^2,\beta^1}(\bar x),
  \end{eqnarray}
  where for $(\beta^1,\beta^2) \in \mathcal{P}(I^{00}(\bar x))$ we define
  \[\nu_i^{\beta^1,\beta^2}(\bar x) := \left\{
    \begin{array}{ll}
      N_{P^1}(F_i(\bar x)) & \textrm{if } i \in I^{0+}(\bar x) \cup \beta^1, \\
      N_{P^2}(F_i(\bar x)) & \textrm{if } i \in I^{0-}(\bar x) \cup I^{+0}(\bar x) \cup I^{+-}(\bar x) \cup \beta^2.
    \end{array} \right.\]
  Note also that for every $i \in V$ we have
  \begin{equation} \label{eq : MstatConvexPiece}
    \nu_i^{I^{00}(\bar x),\emptyset}(\bar x) \subset N_{P}(F_i(\bar x)).
  \end{equation}
  
\section{Solving the auxiliary problem}

  In this section, we describe an algorithm for solving quadratic problems with vanishing constraints of the type
    \begin{equation} \label{eq : deltaOldprob}
      \begin{array}{lrll}
	QPVC(\rho) & \min\limits_{(s,\delta) \in \R^{n+1}} & \frac{1}{2} s^T B s + \nabla f s + \rho ( \frac{1}{2} \delta^2 + \delta) & \\
	& \textrm{subject to } & (1 - \delta) h_i + \nabla h_i s = 0 & i \in E, \\
	&& (1 - \Beta_i^g \delta) g_i + \nabla g_i s \leq 0 & i \in I, \\
	&& (1 - \Beta_i^H \delta) H_i + \nabla H_i s \geq 0, & \\
	&& \left( (1 - \Beta_i^G \delta) G_i + \nabla G_i s \right) \ \left( (1 - \Beta_i^H \delta) H_i + \nabla H_i s \right) \leq 0 & i \in V, \\
	&& - \delta \leq 0. &
      \end{array}
    \end{equation}
  Here the vector $\Beta = (\Beta^g,\Beta^G,\Beta^H) \in \{ 0, 1 \}^{\vert I \vert + 2 \vert V \vert} =: \mathcal{B}$ 
  is chosen at the beginning of the algorithm such that some feasible point is known in advance, e.g. $(s,\delta)=(0,1)$.
  The parameter $\rho$ has to be chosen sufficiently large and acts like a penalty parameter forcing
  $\delta$ to be near zero at the solution. $B$ is a symmetric positive definite $n \times n$ matrix, $\nabla f$, $\nabla h_i$, $\nabla g_i$,
  $\nabla G_i$, $\nabla H_i$ denote row vectors in $\R^n$ and $h_i,g_i,G_i,H_i$ are real numbers.
  Note that this problem is a special case of problem \eqref{eq : genproblem} and consequently the definition
  of $\mathcal{Q}-$ and $\mathcal{Q}_M-$ stationarity as well as the definition of index sets \eqref{eqn : IndexStes}
  remain valid.
  
  It turns out to be much more convenient to operate with a more general notation. Let us denote by $F_i:=(-H_i,G_i)^T$ a vector in $\R^2$,
  by $\nabla F_i := (-\nabla H_i^T,\nabla G_i^T)^T$ a $2 \times n$ matrix and by $P^1 := \{0\} \times \R$ and $P^2:= \R^2_-$ two subsets of $\R^2$.
  Note that for $P$ given by \eqref{eqn : FPdef} it holds that $P = P^1 \cup P^2$.
  The problem \eqref{eq : deltaOldprob} can now be equivalently rewritten in a form
  \begin{equation} \label{eq : deltaprob}
      \begin{array}{lrll}
	QPVC(\rho) & \min\limits_{(s,\delta) \in \R^{n+1}} & \frac{1}{2} s^T B s + \nabla f s + \rho ( \frac{1}{2} \delta^2 + \delta) & \\
	& \textrm{subject to } & (1 - \delta) h_i + \nabla h_i s = 0 & i \in E, \\
	&& (1 - \Beta_i^g \delta) g_i + \nabla g_i s \leq 0 & i \in I, \\
	&& \delta (\Beta_i^H H_i, - \Beta_i^G G_i)^T + F_i + \nabla F_i s \in P & i \in V, \\
	&& - \delta \leq 0. &
      \end{array}
    \end{equation}
  For a given feasible point $(s,\delta)$ for the problem $QPVC(\rho)$ we define the following index sets
    \begin{eqnarray*}
      I^{1}(s,\delta) & := & \{ i \in V \mv \delta (\Beta_i^H H_i, - \Beta_i^G G_i)^T + F_i + \nabla F_i s \in P^1 \setminus P^2 \} =
      I^{0+}(s,\delta), \\
      I^{2}(s,\delta) & := & \{ i \in V \mv \delta (\Beta_i^H H_i, - \Beta_i^G G_i)^T + F_i + \nabla F_i s \in P^2 \setminus P^1 \} =
      I^{+0}(s,\delta) \cup I^{+-}(s,\delta), \\
      I^{0}(s,\delta) & := & \{ i \in V \mv \delta (\Beta_i^H H_i, - \Beta_i^G G_i)^T + F_i + \nabla F_i s \in P^1 \cap P^2 \} =
      I^{0-}(s,\delta) \cup I^{00}(s,\delta),
    \end{eqnarray*}
  where the index sets $I^{0+}(s,\delta)$, $I^{+0}(s,\delta)$, $I^{+-}(s,\delta)$, $I^{0-}(s,\delta)$, $I^{00}(s,\delta)$
  are given by \eqref{eqn : IndexStes}.

  Further, consider the distance function $d$ defined by
  \[d(x,A) := \inf_{y \in A} \norm{x-y}_1,\]
  for $x \in \R^2$ and $A \subset \R^2$.
  The following proposition summarizes some well-known properties of $d$.
  \begin{proposition}
    Let $x \in \R^2$ and $A \subset \R^2$.
      \begin{enumerate}
       \item Let $B \subset \R^2$, then
	     \begin{equation} \label{eq : DistOfCup}
	      d(x, A \cup B) = \min\{ d(x, A), d(x, B) \}.
             \end{equation}
             In particular,
	      \begin{equation} \label{eq : DistToP}
	      d(x,P^1) = (x_1)^+ + (-x_1)^+, \,\, d(x,P^2) = (x_1)^+ + (x_2)^+, \,\,
	      d(x,P) = (x_1)^+ + (\min\{-x_1,x_2\})^+.
	      \end{equation}
       \item $d(\cdot,A) : \R^2 \rightarrow \R^+$ is Lipschitz continuous with Lipschitz modulus $L = 1$ and consequently
	      \begin{equation} \label{eq : DistIfIn}
		d(x,A) \leq d(x+y,A) + \norm{y}_1.
	      \end{equation}
       \item $d(\cdot,A) : \R^2 \rightarrow \R^+$ is convex, provided $A$ is convex.
      \end{enumerate}
  \end{proposition}
  
  Due to the disjunctive structure of the auxiliary problem we can subdivide it into several QP-pieces.
  For every partition $(V_1,V_2) \in \mathcal{P}(V)$ we define the convex quadratic problem
  \begin{equation} \label{eq : deltaMprob}
      \begin{array}{lrll}
	QP(\rho, V_1) & \min\limits_{(s,\delta) \in \mathbb{R}^{n+1}} & \frac{1}{2} s^T B s + \nabla f s + \rho (\frac{1}{2} \delta^2 + \delta) & \\
	& \textrm{subject to } & (1 - \delta) h_i + \nabla h_i s = 0 & i \in E, \\
	&& (1 - \Beta_i^g \delta) g_i + \nabla g_i s \leq 0 & i \in I, \\
	&& \delta (\Beta_i^H H_i, - \Beta_i^G G_i)^T + F_i + \nabla F_i s \in P^1 & i \in V_1, \\
	&& \delta (\Beta_i^H H_i, - \Beta_i^G G_i)^T + F_i + \nabla F_i s \in P^2 & i \in V_2, \\
	&& - \delta \leq 0. &
      \end{array}
  \end{equation}
  Since $(V_1,V_2)$ form a partition of $V$ it is sufficient to define $V_1$ since $V_2$ is given by $V_2 = V \setminus V_1$.
  
  At the solution $(s,\delta)$ of $QP(\rho, V_1)$ there is a corresponding multiplier $\lambda(\rho, V_1) = (\lambda^h,\lambda^g,\lambda^H,\lambda^G)$ 
  and a number $\lambda^{\delta} \geq 0$ with $\lambda^{\delta} \delta = 0$ fulfilling the KKT conditions:
  \begin{eqnarray} \label{eqn : FirstOrder1}
      B s + \nabla f^T + \sum_{i \in E} \lambda_i^h \nabla h_i^T + \sum_{i \in I} \lambda_i^g \nabla g_i^T 
      + \sum_{i \in V} \nabla F_i^T \lambda_i^F & = & 0, \\ \label{eqn : FirstOrder2}
      \rho (\delta + 1) - \lambda^{\delta} - \sum_{i \in E} \lambda_i^h h_i - \sum_{i \in I} \lambda_i^g \Beta_i^g g_i 
      + \sum_{i \in V} (\Beta_i^H H_i, - \Beta_i^G G_i) \lambda_i^F & = & 0, \\
      \label{eqn : ComplemCon1}
      \lambda_i^g ((1 - \Beta_i^g \delta) g_i + \nabla g_i s) = 0, \,\, \lambda_i^g \geq 0, && i \in I, \\ \label{eqn : ComplemCon2}
      \lambda_i^F \in N_{P^1}(\delta (\Beta_i^H H_i, - \Beta_i^G G_i)^T + F_i + \nabla F_i s), && i \in V_1, \\ \label{eqn : ComplemCon3}
      \lambda_i^F \in N_{P^2}(\delta (\Beta_i^H H_i, - \Beta_i^G G_i)^T + F_i + \nabla F_i s), && i \in V_2,
    \end{eqnarray}
  where $\lambda_i^F := (\lambda_i^H,\lambda_i^G)^T$ for $i \in V$. Since $P^1$ and $P^2$ are convex sets,
  the above normal cones are given by \eqref{eq : convexNC}.
    
  The definition of the problem $QP(\rho, V_1)$ allows the following interpretation of $\mathcal{Q}$-stationarity,
  which is a direct consequence of \eqref{eqn : QstatLNC1} and \eqref{eqn : QstatLNC2}.
  \begin{lemma} \label{Lem : QstatQPpiece}
    A point $(s,\delta)$ is $\mathcal{Q}$-stationary with respect to $(\beta^1,\beta^2) \in \mathcal{P}(I^{00}(s,\delta))$ for \eqref{eq : deltaprob}
    if and only if it is the solution of the convex
    problems $QP(\rho, I^{1}(s,\delta) \cup \beta^1)$ and $QP(\rho, I^{1}(s,\delta) \cup \beta^2)$.
  \end{lemma}
  Moreover, since for $V_1 = I^{1}(s,\delta) \cup I^{00}(s,\delta)$ the conditions \eqref{eqn : ComplemCon2},\eqref{eqn : ComplemCon3}
  read as $\lambda_i^F \in \nu_i^{I^{00}(s,\delta),\emptyset}(s,\delta)$, it follows from \eqref{eq : MstatConvexPiece}
  that if a point $(s,\delta)$ is the solution of $QP(\rho, I^{1}(s,\delta) \cup I^{00}(s,\delta))$ then it is M-stationary for \eqref{eq : deltaprob}.
  
  Finally, let us denote by $\bar\delta(V_1)$ the objective value at a solution of the problem
  \begin{equation} \label{eq : MinDelProb}
    \min_{(s,\delta) \in \mathbb{R}^{n+1}} \,\, \delta \qquad \textrm{ subject to the constraints of \eqref{eq : deltaMprob}.}
  \end{equation}
  An outline of the algorithm for solving $QPVC(\rho)$ is as follows.
  
    \begin{algorithm}[Solving the QPVC] \label{AlgSol} \rm \mbox{}
    
    Let $\zeta \in (0,1)$, $\bar{\rho} > 1$ and $\rho > 0$ be given.
    
    \Itl1{1:} Initialize:
    \Itl2{}   Set the starting point $(s^0, \delta^0) := (0,1)$, define the vector $\Beta$ by
    \Itl3{}   \begin{equation} \label{eq : DefBet}
		\Beta_i^{g} := \left\{
		\begin{array}{ll}
		  1 & \quad \textrm{if } g_i > 0, \\
		  0 & \quad \textrm{if } g_i \leq 0,
		\end{array} \right. \qquad
		(\Beta_i^{H}, \Beta_i^{G}) := \left\{
		\begin{array}{ll}
		  (0,0) & \quad \textrm{if } d(F_i,P) = 0, \\
		  (1,0) & \quad \textrm{if } 0 < d(F_i,P^1) \leq d(F_i,P^2), \\
		  (0,1) & \quad \textrm{if } 0 < d(F_i,P^2) < d(F_i,P^1)
		\end{array} \right.
	      \end{equation}
    \Itl3{}   and set the partition $V_1^1 := I^{1}(s^0,\delta^0)$ and the counter of pieces $t:=0$.
    \Itl2{}   Compute $(s^{1}, \delta^{1})$ as the solution and $\lambda^{1}$ as the corresponding multiplier
    \Itl3{}   of the convex problem $QP(\rho, V_1^{1})$ and set $t:=1$.
    \Itl2{}   If $\delta^1 > \delta^0$, perform a restart: set $\rho := \rho \bar \rho$ and go to step 1.
    \Itl1{2:} Improvement step:
    \Itl2{}   {\tt while} $(s^t,\delta^t)$ is not a solution of the following four convex problems:
	      \begin{eqnarray} \label{eqn : b1b2Stat}
		QP(\rho, I^{1}(s^t,\delta^t) \cup (I^{00}(s^t,\delta^t) \cap V_1^t)), &&
		QP(\rho, I^{1}(s^t,\delta^t) \cup (I^{00}(s^t,\delta^t) \setminus V_1^t)), \\
		\label{eqn : I00EmptyStat}
		QP(\rho, I^{1}(s^t,\delta^t)), &&
		QP(\rho, I^{1}(s^t,\delta^t) \cup I^{00}(s^t,\delta^t)).
	      \end{eqnarray}
    \Itl3{}   Compute $(s^{t+1}, \delta^{t+1})$ as the solution and $\lambda^{t+1}$ as the corresponding multiplier
    \Itl4{}   of the first problem with $(s^{t+1}, \delta^{t+1}) \neq (s^{t}, \delta^{t})$, set $V_1^{t+1}$ to the
    \Itl4{}   corresponding index set and increase the counter $t$ of pieces by $1$.
    \Itl3{}   If $\delta^t > \delta^{t-1}$, perform a restart: set $\rho := \rho \bar \rho$ and go to step 1.
    \Itl1{3:} Check for successful termination:
    \Itl2{}   If $\delta^t < \zeta$ set $N:=t$, stop the algorithm and return.
    \Itl1{4:} Check the degeneracy:
    \Itl2{}   If the non-degeneracy condition
	      \begin{equation} \label{eq : NonDegen}
		\min \{ \bar\delta(I^{1}(s^t,\delta^t)), \bar\delta(I^{1}(s^t,\delta^t) \cup I^{00}(s^t,\delta^t)) \} < \zeta
	      \end{equation}
    \Itl3{}   is fulfilled, perform a restart: set $\rho := \rho \bar \rho$ and go to step 1.
    \Itl2{}   Else stop the algorithm because of degeneracy. 
    \end{algorithm}
  
  The selection of the index sets in step 2 is motivated by Lemma \ref{Lem : QstatQPpiece},
  since if $(s,\delta)$ is the solution of convex problems \eqref{eqn : b1b2Stat}, then it is
  $\mathcal{Q}$-stationary and if $(s,\delta)$ is also the
  solution of convex problems \eqref{eqn : I00EmptyStat}, then it is even $\mathcal{Q}_M$-stationary
  for problem \eqref{eq : deltaprob}.
  
  We first summarize some consequences of the Initialization step.
  \begin{proposition} \label{Prop : InitStep}
  \begin{enumerate}
   \item Vector $\Beta$ is chosen in a way that for all $i \in V$ it holds that
      \begin{equation} \label{eq : betaeff}
	\norm{(\Beta_i^H H_i, - \Beta_i^G G_i)^T}_1 = d(F_i, P).
      \end{equation}
   \item Partition $(V_1^1,V_2^1)$ is chosen in a way that for $j=1,2$ it holds that
   \begin{equation} \label{eq : InitPart}
    i \in V_j^1 \, \textrm{ implies } \, d(F_i, P) = d(F_i, P^j).
   \end{equation}
  \end{enumerate}
  \end{proposition}
  \begin{proof}
    1. If $d(F_i, P) = 0$ we have $(\Beta_i^H, \Beta_i^G) = (0,0)$
    and \eqref{eq : betaeff} obviously holds. If $0 < d(F_i,P^1) \leq d(F_i,P^2)$ we have $(\Beta_i^H, \Beta_i^G) = (1,0)$ and we obtain
    \[\norm{(\Beta_i^H H_i, - \Beta_i^G G_i)^T}_1 = \vert H_i \vert = d(F_i,P^1) = d(F_i,P) \]
    by \eqref{eq : DistToP} and \eqref{eq : DistOfCup}. Finally, if $0 < d(F_i,P^2) < d(F_i,P^1)$ we have
    $H_i < 0 < G_i$, $(\Beta_i^H, \Beta_i^G) = (0,1)$ and thus
    \[\norm{(\Beta_i^H H_i, - \Beta_i^G G_i)^T}_1 = \vert G_i \vert = (H_i)^+ + (G_i)^+ = d(F_i,P^2) = d(F_i,P) \]
    follows again by \eqref{eq : DistToP} and \eqref{eq : DistOfCup}.
    
    2. If $(\Beta_i^H H_i, - \Beta_i^G G_i)^T + F_i \in P^j$ for some $i \in V$ and $j = 1,2$,
    by \eqref{eq : DistIfIn} and \eqref{eq : betaeff} we obtain
    \[d(F_i,P^j) \leq \norm{(\Beta_i^H H_i, - \Beta_i^G G_i)^T}_1 = d(F_i, P) \]
    and consequently $d(F_i,P^j) = d(F_i,P)$, because of \eqref{eq : DistOfCup}.
    Hence we conclude that $i \in (I^{j}(s^0,\delta^0) \cup I^{0}(s^0,\delta^0))$ implies $d(F_i,P^j) = d(F_i,P)$ for $j = 1,2$
    and the statement now follows from the fact that $V_1^1 = I^{1}(s^0,\delta^0)$ and $V_2^1 = I^{2}(s^0,\delta^0) \cup I^{0}(s^0,\delta^0)$.
  \end{proof}
    
  The following lemma plays a crucial part in proving the finiteness of the Algorithm \ref{AlgSol}.
  
  \begin{lemma} \label{Lem : RhoMinDelta}
    For each partition $(V_1,V_2) \in \mathcal{P}(V)$ there exists a positive constant $C_{\rho}(V_1)$ such that for every
    $\rho \geq C_{\rho}(V_1)$ the solution $(s,\delta)$ of $QP(\rho,V_1)$ fulfills $\delta = \bar\delta(V_1)$.
  \end{lemma}
  \begin{proof}
    Let $(s(V_1),\delta(V_1))$ denote a solution of \eqref{eq : MinDelProb}. Since $\delta(V_1)=\bar\delta(V_1)$, it follows that the problem
    \begin{equation} \label{eq : deltaZeroMprob}
      \begin{array}{rll}
	\min\limits_{s \in \mathbb{R}^{n}} & \frac{1}{2} s^T B s + \nabla f s & \\
	\textrm{subject to } & (1 - \bar\delta(V_1)) h_i + \nabla h_i s = 0 & i \in E, \\
	& (1 - \Beta_i^g \bar\delta(V_1)) g_i + \nabla g_i s \leq 0 & i \in I, \\
	& \bar\delta(V_1) (\Beta_i^H H_i, - \Beta_i^G G_i)^T + F_i + \nabla F_i s \in P^1 & i \in V_1, \\
	& \bar\delta(V_1) (\Beta_i^H H_i, - \Beta_i^G G_i)^T + F_i + \nabla F_i s \in P^2 & i \in V_2
      \end{array}
    \end{equation}
    is feasible and by $\bar s(V_1)$ we denote the solution of this problem and by $\bar \lambda(V_1)$ the corresponding multiplier.
    Further, $(\bar s(V_1),\bar\delta(V_1))$ is a solution of \eqref{eq : MinDelProb} and by $\lambda(V_1)$
    we denote the corresponding multiplier.
    
    Then, triple $(\bar s(V_1),\bar\delta(V_1))$ and $\bar \lambda(V_1)$ fulfills \eqref{eqn : FirstOrder1}
    and \eqref{eqn : ComplemCon1}-\eqref{eqn : ComplemCon3}.
    Moreover, triple $(\bar s(V_1),\bar\delta(V_1))$ and $\lambda(V_1)$ fulfills \eqref{eqn : ComplemCon1}-\eqref{eqn : ComplemCon3} and
    \begin{eqnarray} \label{eqn : FirstOrder1Small}
      \sum_{i \in E} \lambda(V_1)_i^h \nabla h_i^T + \sum_{i \in I} \lambda(V_1)_i^g \nabla g_i^T 
      + \sum_{i \in V} \nabla F_i^T \lambda(V_1)_i^F & = & 0, \\ \label{eqn : FirstOrder2Small}
      1 - \lambda^{\delta} - \sum_{i \in E} \lambda(V_1)_i^h h_i - \sum_{i \in I} \lambda(V_1)_i^g \Beta_i^g g_i 
      + \sum_{i \in V} (\Beta_i^H H_i, - \Beta_i^G G_i) \lambda(V_1)_i^F & = & 0
    \end{eqnarray}
    for some $\lambda^{\delta} \geq 0$ with $\lambda^{\delta} \bar\delta(V_1) = 0$.
    
    Let $C_{\rho}(V_1)$ be a positive constant such that for all $\rho \geq C_{\rho}(V_1)$ we have
    \[\alpha := \rho(\bar\delta(V_1) + 1) - \sum_{i \in E} \bar \lambda(V_1)_i^h h_i - \sum_{i \in I} \bar \lambda(V_1)_i^g \Beta_i^g g_i 
      + \sum_{i \in V} (\Beta_i^H H_i, - \Beta_i^G G_i) \bar \lambda(V_1)_i^F \geq 0\]
    and set $\tilde \lambda^{\delta} := \alpha \lambda^{\delta} \geq 0$ and $\tilde \lambda := \bar \lambda(V_1) + \alpha \lambda(V_1)$.
    We will now show that for such $\rho$ it holds that $(\bar s(V_1),\bar\delta(V_1))$ is the solution of $QP(\rho,V_1)$.
    
    Clearly, $\tilde \lambda^{\delta} \bar\delta(V_1) = \alpha \lambda^{\delta} \bar\delta(V_1) = 0$ and the triple
    $(\bar s(V_1),\bar\delta(V_1))$ and $\tilde \lambda$ also fulfills \eqref{eqn : FirstOrder1} due to \eqref{eqn : FirstOrder1Small}
    and it fulfills \eqref{eqn : ComplemCon1}-\eqref{eqn : ComplemCon3} due to the convexity of the normal cones. Moreover, taking into account
    the definitions of $\alpha$, $\tilde \lambda^{\delta}$ and $\tilde \lambda$ together with \eqref{eqn : FirstOrder2Small}, we obtain
    \[\rho (\bar\delta(V_1) + 1) - \tilde \lambda^{\delta} - \sum_{i \in E} \tilde \lambda_i^h h_i - \sum_{i \in I} \tilde \lambda_i^g \Beta_i^g g_i 
      + \sum_{i \in V} (\Beta_i^H H_i, - \Beta_i^G G_i) \tilde \lambda_i^F =
      \alpha - \alpha \lambda^{\delta} - \alpha(1 - \lambda^{\delta}) = 0,\]
    showing also \eqref{eqn : FirstOrder2}. Hence $(\bar s(V_1),\bar\delta(V_1))$ is the solution of $QP(\rho,V_1)$ and the proof is complete.
  \end{proof}
  
  We now formulate the main theorem of this section.
  
  \begin{theorem} \label{The : FinitAndFinal}
    \begin{enumerate}
     \item Algorithm \ref{AlgSol} is finite.
     \item If the Algorithm \ref{AlgSol} is not terminated because of degeneracy, then $(s^N,\delta^N)$ is $\mathcal{Q}_M$-stationary
	   for the problem \eqref{eq : deltaprob} and $\delta^N < \zeta$.
    \end{enumerate}
    \end{theorem}
    \begin{proof}
      1. The algorithm is obviously finite unless we perform a restart and hence increase $\rho$. Thus we can assume that $\rho$ is sufficiently
      large, say \[\rho \geq C_{\rho} := \max_{(V_1,V_2) \in \mathcal P(V)} C_{\rho}(V_1),\]
      with $C_{\rho}(V_1)$ given by the previous lemma. However this means, taking into account also Proposition \ref{Prop : Props} (1.),
      that $(s^{t-1},\delta^{t-1})$ is feasible for the problem $QP(\rho,V_1^t)$ for all $t$, hence $\delta^{t-1} \geq \bar\delta(V_1^t)$
      and $(s^t,\delta^t)$ is the solution of $QP(\rho,V_1^t)$, implying $\delta^{t} = \bar\delta(V_1^t)$
      and consequently $\delta^{t} \leq \delta^{t-1}$. Therefore we do not perform a restart in step 1 or step 2.
      On the other hand, since we enter steps 3 and 4 with $\delta^t = \bar\delta(I^{1}(s^t,\delta^t)) = \bar\delta(I^{1}(s^t,\delta^t) \cup I^{00}(s^t,\delta^t))$,
      we either terminate the algorithm in step 3 with $\delta^t < \zeta$ if the non-degeneracy condition \eqref{eq : NonDegen} is fulfilled
      or we terminate the algorithm because of degeneracy in step 4. This finishes the proof.
      
      2. The statement regarding stationarity follows easily from the fact that we enter step 3 of the algorithm only when $(s,\delta)$ is a solution of
      problems \eqref{eqn : I00EmptyStat} and this means that it is also $\mathcal{Q}$-stationary with respect to
      $(\emptyset,I^{00}(s^N,\delta^N))$ by Lemma \ref{Lem : QstatQPpiece}. Thus, $(s,\delta)$ is also
      $\mathcal{Q}_M$-stationary for problem \eqref{eq : deltaprob}. The claim about $\delta$ follows
      from the assumption that the Algorithm \ref{AlgSol} is not terminated because of degeneracy.
    \end{proof}
    
      We conclude this section with the following proposition that brings together the basic properties of the Algorithm \ref{AlgSol}.
    
    \begin{proposition} \label{Prop : Props}
      If the Algorithm \ref{AlgSol} is not terminated because of degeneracy, then the following properties hold:
      \begin{enumerate}
	\item For all $t = 1, \ldots, N$ the points $(s^{t-1},\delta^{t-1})$ and $(s^{t},\delta^{t})$ are feasible
	      for the problem $QP(\rho,V_1^t)$ and the point $(s^t,\delta^t)$ is also the solution of the convex problem $QP(\rho,V_1^t)$.
	\item For all $t = 1, \ldots, N$ it holds that
	      \begin{equation}
		0 \leq \delta^{t} \leq \delta^{t-1} \leq 1.
	      \end{equation}
	\item There exists a constant $C_t$, dependent only on the number of constraints, such that
		\begin{equation} \label{eq : StepsBound}
		    N \leq C_t.
		\end{equation}
      \end{enumerate}      
    \end{proposition}
    \begin{proof}   
      1. By definitions of the problems $QPVC(\rho)$ and $QP(\rho,V_1)$ it follows that a point $(s,\delta)$, feasible
      for $QPVC(\rho)$, is feasible for $QP(\rho,V_1)$ if and only if
      \begin{equation} \label{eq : Feasibility}
	I^1(s,\delta) \subset V_1 \subset I^1(s,\delta) \cup I^{0}(s,\delta).
      \end{equation}
      The point $(s^0,\delta^0)$ is clearly feasible for $QP(\rho,V_1^1)$
      and similarly the point $(s^{t},\delta^{t})$ is feasible for $QP(\rho,V_1^{t+1})$ for all $t = 1, \ldots, N-1$,
      since the partition $V_1^{t+1}$ is defined by one of the index sets of \eqref{eqn : b1b2Stat}-\eqref{eqn : I00EmptyStat}
      and thus fulfills \eqref{eq : Feasibility}. However, feasibility of $(s^{t+1},\delta^{t+1})$ for $QP(\rho,V_1^{t+1})$,
      together with $(s^{t+1},\delta^{t+1})$ being the solution of $QP(\rho,V_1^{t+1})$, then follows from its definition.
      
      2. Statement follows from $\delta^0 = 1$, from the fact that we perform a restart whenever $\delta^t > \delta^{t-1}$ occurs
      and from the constraint $-\delta \leq 0$.
  
      3. Since whenever the parameter $\rho$ is increased the algorithm goes to the step 1 and thus the counter $t$ of the pieces
      is reset to $0$, it follows that after the last time the algorithm enters step 1 we keep $\rho$ constant.
      It is obvious that all the index sets $V_1^t$ are pairwise different implying
      that the maximum of switches to a new piece is $2^{\vert V \vert}$.
    \end{proof}

\section{The basic SQP algorithm for MPVC}

  An outline of the basic algorithm is as follows.
  
  \begin{algorithm}[Solving the MPVC] \label{AlgMPCC} \rm \mbox{}
  
    \Itl1{1:} Initialization:
    \Itl2{}   Select a starting point $x_0 \in \R^n$ together with a positive definite $n \times n$ matrix $B_0$,
    \Itl3{}   a parameter $\rho_0 > 0$ and constants $\zeta \in (0,1)$ and $\bar\rho > 1$.
    \Itl2{}   Select positive penalty parameters $\sigma_{-1} = (\sigma^h_{-1}, \sigma^g_{-1}, \sigma^F_{-1})$.
    \Itl2{}   Set the iteration counter $k := 0$.
    \Itl1{2:} Solve the Auxiliary problem:
    \Itl2{}   Run Algorithm \ref{AlgSol} with data $\zeta, \bar\rho, \rho:= \rho_k, B:=B_k, \nabla f := \nabla f (x_k),$
    \Itl3{}   $h_i := h_i(x_k), \nabla h_i := \nabla h_i (x_k), i \in E,$ etc.
    \Itl2{}   If the Algorithm \ref{AlgSol} stops because of degeneracy,
    \Itl3{}   stop the Algorithm \ref{AlgMPCC} with an error message.
    \Itl2{}   If the final iterate $s^N$ is zero, stop the Algorithm \ref{AlgMPCC} and return $x_k$ as a solution.
    \Itl1{3:} Next iterate:
    \Itl2{}   Compute new penalty parameters $\sigma_{k}$.
    \Itl2{}   Set $x_{k+1} := x_k + s_k$ where $s_k$ is a point on the polygonal line connecting the points
    \Itl3{}   $s^0,s^1, \ldots, s^N$ such that an appropriate merit function depending on $\sigma_{k}$ is decreased.   
    \Itl2{}   Set $\rho_{k+1} := \rho$, the final value of $\rho$ in Algorithm \ref{AlgSol}.
    \Itl2{}   Update $B_{k}$ to get positive definite matrix $B_{k+1}$.
    \Itl2{}   Set $k := k+1$ and go to step 2.
  \end{algorithm}
  
  \begin{remark} \label{rem : StoppingCriteria}
    We terminate the Algorithm \ref{AlgMPCC} only in the following two cases. In the first case
    no sufficient reduction of the violation of the constraints can be achieved.
    The second case will be satisfied only by chance when the current iterate is a $\mathcal{Q}_M$-stationary solution.
    Normally, this algorithm produces an infinite sequence of iterates and we must include a stopping criterion
    for convergence. Such a criterion could be that the violation of the constraints at some iterate is sufficiently small,
    \[\max \{ \max_{i \in E} \vert h_i(x_k) \vert, \max_{i \in I} (g_i(x_k))^+, \max_{i \in V} d(F_i(x_k),P) \} \leq \epsilon_C,\]
    where $F_i$ is given by \eqref{eqn : FPdef} and the expected decrease in our merit function is sufficiently small,
    \[ (s_k^{N_k})^T B_k s_k^{N_k} \leq \epsilon_1, \]
    see Proposition \ref{pro : MainMerit} below.  
  \end{remark}

  \subsection{The next iterate}
  
    Denote the outcome of Algorithm \ref{AlgSol} at the $k-$th iterate by
    $$(s_k^{t}, \delta_k^{t}), \lambda_k^{t}, (V_{1,k}^{t}, V_{2,k}^{t}) \textrm{ for } t = 0, \ldots, N_k \textrm{ and }
    \Beta_k, \underline\lambda_k^{N_k}, \overline\lambda_k^{N_k}.$$
    The new penalty parameters are computed by
    \begin{eqnarray} \label{eqn : PenalParam}
      \sigma_{i,k}^h = \begin{cases}
			    \xi_2 \tilde\lambda_{i,k}^{h} & \textrm{ if } \sigma_{i,k-1}^h < \xi_1 \tilde\lambda_{i,k}^{h}, \\
			    \sigma_{i,k-1}^h & \textrm{ else},
                          \end{cases} \qquad
       \sigma_{i,k}^g = \begin{cases}
			    \xi_2 \tilde\lambda_{i,k}^{g} & \textrm{ if } \sigma_{i,k-1}^g < \xi_1 \tilde\lambda_{i,k}^{g}, \\
			    \sigma_{i,k-1}^g & \textrm{ else},
                          \end{cases}
                          \\ \nonumber
       \sigma_{i,k}^F = \begin{cases}
			    \xi_2 \tilde{\lambda}_{i,k}^F & \textrm{ if }
			    \sigma_{i,k-1}^F < \xi_1 \tilde{\lambda}_{i,k-1}^F, \\
			    \sigma_{i,k-1}^F & \textrm{ else},
                          \end{cases}
    \end{eqnarray}
    where
    \begin{equation} \label{eq : DefTilLamb}
      \tilde{\lambda}_{i,k}^h = \max \vert \lambda_{i,k}^{h,t} \vert, \quad
      \tilde{\lambda}_{i,k}^g = \max \vert \lambda_{i,k}^{g,t} \vert, \quad
      \tilde{\lambda}_{i,k}^F = \max \norm{\lambda_{i,k}^{F,t}}_{\infty},
    \end{equation}
    with maximum being taken over $t \in \{ 1, \ldots, N_k \}$
    and $1 < \xi_1 < \xi_2$. Note that this choice of $\sigma_{k}$ ensures
    \begin{equation}
      \label{EqPropSigma}
      \sigma_{k}^h \geq \tilde\lambda_k^h,\ \sigma_{k}^g \geq \tilde\lambda_k^g,\ \sigma_{k}^F \geq \tilde{\lambda}_{k}^F.
    \end{equation}
    
    \subsubsection{The merit function}
            
      We are looking for the next iterate at the polygonal line connecting the points $s_k^0,s_k^1, \ldots, s_k^{N_k}$.
      For each line segment $[s_k^{t-1},s_k^t] := \{(1-\alpha) s_k^{t-1} + \alpha s_k^t \mv \alpha \in [0,1] \}, t = 1, \ldots, N_k$
      we consider the functions
      \begin{eqnarray*}
	\phi_k^t(\alpha) & := & f(x_k + s) + \sum \limits_{i \in E} \sigma_{i,k}^h \vert h_i(x_k + s) \vert +
	  \sum \limits_{i \in I} \sigma_{i,k}^g ( g_i(x_k + s) )^+ \\
	  && + \sum \limits_{i \in V_{1,k}^t} \sigma_{i,k}^F d(F_i(x_k + s),P^1)
	  + \sum \limits_{i \in V_{2,k}^t} \sigma_{i,k}^F d(F_i(x_k + s),P^2), \\
	\hat{\phi}_k^t(\alpha) & := & f + \nabla f s + \frac{1}{2} s^TB_ks +
	  \sum \limits_{i \in E} \sigma_{i,k}^h \vert h_i + \nabla h_i s \vert +
	  \sum \limits_{i \in I} \sigma_{i,k}^g ( g_i + \nabla g_i s )^+ \\
	  && + \sum \limits_{i \in V_{1,k}^t} \sigma_{i,k}^F d(F_i + \nabla F_i s,P^1)
	  + \sum \limits_{i \in V_{2,k}^t} \sigma_{i,k}^F d(F_i + \nabla F_i s,P^2),
      \end{eqnarray*}
      where $s = (1 - \alpha) s_k^{t-1} + \alpha s_k^t$ and $f = f(x_k)$, $\nabla f = \nabla f (x_k)$,
      $h_i = h_i(x_k), \nabla h_i = \nabla h_i (x_k), i \in E,$ etc.
      and we further denote
      \begin{equation} \label{eq : DefOfR}
	r^{t}_{k,0}:= \hat\phi_k^{t}(0) - \hat\phi_k^{1}(0), \quad r^t_{k,1} := \hat\phi_k^{t}(1) - \hat\phi_k^{1}(0).
      \end{equation}
      
      \begin{lemma} \label{lem : convex}
      \begin{enumerate}
       \item For every $t \in \{ 1, \ldots, N_k \}$ the function $\hat{\phi}_k^t$ is convex.
       \item For every $t \in \{ 1, \ldots, N_k \}$ the function $\hat{\phi}_k^t$ is a first order approximation of $\phi_k^t$, that is
      \[\vert \phi_k^t(\alpha) - \hat{\phi}_k^t(\alpha) \vert = o(\norm{s}),\]
      where $s = (1 - \alpha) s_k^{t-1} + \alpha s_k^t$.
      \end{enumerate}
      \end{lemma}
      \begin{proof}
       1. By convexity of $P^1$ and $P^2$, $\hat{\phi}_k^t$ is convex because it is sum of convex functions.
       
       2. By Lipschitz continuity of distance function with Lipschitz modulus $L = 1$ we conclude
       \begin{eqnarray*}
        \vert \phi_k^t(\alpha) - \hat{\phi}_k^t(\alpha) \vert & \leq &
	  \vert f(x_k + s) - f - \nabla f s - \frac{1}{2} s^TB_ks \vert +
	  \sum \limits_{i \in E} \sigma_{i,k}^h \vert h_i(x_k + s) - h_i - \nabla h_i s \vert \\
	  && \sum \limits_{i \in I} \sigma_{i,k}^g \vert g_i(x_k + s) - g_i - \nabla g_i s \vert 
	  + \sum \limits_{i \in V} \sigma_{i,k}^F \norm{ F_i(x_k + s) - F_i - \nabla F_i s}_1
       \end{eqnarray*}
       and hence the assertion follows.
      \end{proof}

      We state now the main result of this subsection. For the sake of simplicity we omit the iteration index $k$ in this part.
      
      \begin{proposition} \label{pro : MainMerit}
	For every $t \in \{ 1, \ldots, N_k \}$
	  \begin{eqnarray} \label{eqn : BoundMain}
	    \hat\phi^{t}(0) - \hat\phi^{1}(0) & \leq &
	    - \sum \limits_{\tau = 1}^{t-1} \frac{1}{2} (s^{\tau} - s^{\tau-1})^T B (s^{\tau} - s^{\tau-1}) \, \, \leq \, \, 0,
	    \\ \label{eqn : BoundMain2}
	    \hat\phi^{t}(1) - \hat\phi^{1}(0) & \leq &
	    - \sum \limits_{\tau = 1}^{t} \frac{1}{2} (s^{\tau} - s^{\tau-1})^T B (s^{\tau} - s^{\tau-1}) \, \, \leq \, \, 0.
	  \end{eqnarray}
      \end{proposition}
    
    \begin{proof}      
      Fix $t \in \{ 1, \ldots, N_k \}$ and note that
      \begin{eqnarray*}
	  1/2 (s^{t})^T B s^{t} + \nabla f s^{t} & = & 1/2 (s^{t})^T B s^{t} + \nabla f s^{t} - 1/2 (s^{0})^T B s^{0} - \nabla f s^{0} \\
	  & = & \sum_{\tau = 1}^t 1/2 (s^{\tau})^T B s^{\tau} - 1/2(s^{\tau-1})^T B s^{\tau-1} + \nabla f (s^{\tau} - s^{\tau-1}),
      \end{eqnarray*}
      because of $s^0 = 0$. For $j=0,1$ consider $r^{t+j}_{1-j}$ defined by \eqref{eq : DefOfR}. We obtain
	  \begin{eqnarray} \label{eqn : Basic}
	    r^{t+j}_{1-j} & = & \sum_{\tau = 1}^t \left( \frac{1}{2} (s^{\tau})^T B s^{\tau} - \frac{1}{2}(s^{\tau-1})^T B s^{\tau-1}
	    + \nabla f (s^{\tau} - s^{\tau-1}) \right) \\ \nonumber
	    && + \sum \limits_{i \in E} \sigma_i^h \left( \vert h_i + \nabla h_i s^{t} \vert - \vert h_i \vert \right)
	    + \sum \limits_{i \in I} \sigma_i^g \left( (g_i + \nabla g_i s^{t})^+ - (g_i)^+ \right) \\ \nonumber
	    && + \sum \limits_{i \in V_1^{t+j}} \sigma_i^F d(F_i + \nabla F_i s^t,P^1)
	    + \sum \limits_{i \in V_2^{t+j}} \sigma_i^F d(F_i + \nabla F_i s^t,P^2) \\ \nonumber
	    && - \sum \limits_{i \in V_1^{1}} \sigma_i^F d(F_i,P^1)
	    - \sum \limits_{i \in V_2^{1}} \sigma_i^F d(F_i,P^2).
	  \end{eqnarray}
      
	  Using that $(s^{\tau},\delta^{\tau})$ is the solution of $QP(\rho,V_1^{\tau})$ and multiplying the first order optimality
	  condition \eqref{eqn : FirstOrder1} by $(s^{\tau} - s^{\tau-1})^T$ yields
	  \begin{equation} \label{eq:mult}
	    (s^{\tau} - s^{\tau-1})^T \left( B s^{\tau} + \nabla f^T + \sum \limits_{i \in E} \lambda_i^{h,\tau} \nabla h_i^T +
	    \sum \limits_{i \in I} \lambda_i^{g,\tau} \nabla g_i^T + \sum \limits_{i \in V} \nabla F_i^T \lambda_i^{F,\tau} \right) = 0.
	  \end{equation}
	  Summing up the expression on the left hand side from $\tau =1$ to $t$, subtracting it from the right hand side of \eqref{eqn : Basic}
	  and taking into account the identity
	  $$1/2 (s^{\tau})^T B s^{\tau} - 1/2(s^{\tau-1})^T B s^{\tau-1} - (s^{\tau} - s^{\tau-1})^T B s^{\tau} =
	  - 1/2 (s^{\tau} - s^{\tau-1})^T B (s^{\tau} - s^{\tau-1})$$
	  we obtain for $j=0,1$
	  \begin{eqnarray} \label{eqn : TheLong}
	    r^{t+j}_{1-j} & = & - \sum \limits_{\tau = 1}^{t} \frac{1}{2} (s^{\tau} - s^{\tau-1})^T B (s^{\tau} - s^{\tau-1}) \\ \nonumber
	    & & + \sum \limits_{i \in E} \left( \sigma_i^h (\vert h_i + \nabla h_i s^{t} \vert - \vert h_i \vert)
	    - \sum \limits_{\tau = 1}^{t} \lambda_i^{h,\tau} \nabla h_i (s^{\tau} - s^{\tau-1}) \right) \\ \nonumber
	    & & + \sum \limits_{i \in I} \left( \sigma_i^g ( (g_i + \nabla g_i s^{t})^+ - (g_i)^+ )
	    - \sum \limits_{\tau = 1}^{t} \lambda_i^{g,\tau} \nabla g_i (s^{\tau} - s^{\tau-1}) \right) \\ \nonumber
	    && + \sum \limits_{i \in V_1^{t+j}} \sigma_i^F d(F_i + \nabla F_i s^t,P^1)
	    + \sum \limits_{i \in V_2^{t+j}} \sigma_i^F d(F_i + \nabla F_i s^t,P^2) \\ \nonumber
	    && - \sum \limits_{i \in V_1^{1}} \sigma_i^F d(F_i,P^1)
	    - \sum \limits_{i \in V_2^{1}} \sigma_i^F d(F_i,P^2)
	    - \sum \limits_{i \in V} \sum \limits_{\tau = 1}^{t} (\lambda_i^{F,\tau})^T \nabla F_i (s^{\tau} - s^{\tau-1}).
	  \end{eqnarray}
	  
	  First, we claim that
	  \begin{equation} \label{eq : MultiplTerm}
	    - \sum \limits_{i \in V} \sum \limits_{\tau = 1}^{t} (\lambda_i^{F,\tau})^T \nabla F_i (s^{\tau} - s^{\tau-1})
	    \leq \sum \limits_{i \in V} \tilde{\lambda}_{i}^F (1 - \delta^t) d(F_i,P).
	  \end{equation}
	  Consider $i \in V$ and $\tau \in \{1, \ldots, t\}$ with $i \in V_1^{\tau}$. By the feasibility of $(s^{\tau},\delta^{\tau})$
	  and $(s^{\tau-1},\delta^{\tau-1})$ for $QP(\rho, V_1^{\tau})$ it follows that
	  \[\delta^{\tau} (\Beta_i^H H_i, - \Beta_i^G G_i)^T + F_i + \nabla F_i s^{\tau} \in P^1, \quad
	    \delta^{\tau-1} (\Beta_i^H H_i, - \Beta_i^G G_i)^T + F_i + \nabla F_i s^{\tau-1} \in P^1\]
	  and hence from \eqref{eqn : ComplemCon2} and \eqref{eq : convexNC} we conclude
	  \[- (\lambda_i^{F,\tau})^T \left( \nabla F_i (s^{\tau} - s^{\tau-1}) + (\delta^{\tau} - \delta^{\tau-1})
	   (\Beta_i^H H_i, - \Beta_i^G G_i)^T \right) \leq 0\]
	  and consequently
	  \begin{equation} \label{eq : MultiplBound}
	  - (\lambda_i^{F,\tau})^T \nabla F_i (s^{\tau} - s^{\tau-1}) \leq (\lambda_i^{F,\tau})^T
	  (\delta^{\tau} - \delta^{\tau-1}) (\Beta_i^H H_i, - \Beta_i^G G_i)^T \leq
	  \tilde{\lambda}_{i}^F (\delta^{\tau-1} - \delta^{\tau}) d(F_i,P)
	  \end{equation}
	  follows by the H\"older inequality and \eqref{eq : betaeff}.  
	  
	  Analogous argumentation yields \eqref{eq : MultiplBound} also for $i, \tau$ with $i \in V_2^{\tau}$
	  and since $V_1^{\tau},V_2^{\tau}$ form a partition of $V$, the claimed inequality \eqref{eq : MultiplTerm} follows.
	  
	  Further, we claim that for $j=0,1$ it holds that
	  \begin{equation} \label{eq : FinalTerm}
	    \sum \limits_{i \in V_1^{t+j}} \sigma_i^F d(F_i + \nabla F_i s^t,P^1)
	    + \sum \limits_{i \in V_2^{t+j}} \sigma_i^F d(F_i + \nabla F_i s^t,P^2) \leq
	    \sum \limits_{i \in V} \sigma_{i}^F \delta^t d(F_i,P).
	  \end{equation}
	  From feasibility of $(s^t,\delta^t)$ for either $QP(\rho,V_1^{t})$ or $QP(\rho,V_1^{t+1})$
	  for $i \in V_1^t \cup V_1^{t+1}$ it follows that
	  \[\delta^{t} (\Beta_i^H H_i, - \Beta_i^G G_i)^T + F_i + \nabla F_i s^{t} \in P^1\]
	  and hence, using \eqref{eq : betaeff} and \eqref{eq : DistIfIn},
	  \begin{equation} \label{eq : FinalBound}
	    \sigma_i^F d(F_i + \nabla F_i s^{t}, P^1) \leq \sigma_i^F \norm{\delta^{t} (\Beta_i^H H_i, - \Beta_i^G G_i)^T}_1 =
	    \sigma_i^F \delta^{t} d(F_i,P).
	  \end{equation}

	  Again, for $i \in V_2^{t}$ or $i \in V_2^{t+1}$ it holds that
	  $\sigma_i^F d(F_i + \nabla F_i s^{t}, P^2) \leq \sigma_i^F \delta^{t} d(F_i,P)$ by analogous argumentation
	  and since $V_1^{t},V_2^{t}$ and $V_1^{t+1},V_2^{t+1}$ form a partition of $V$,
	  the claimed inequality \eqref{eq : FinalTerm} follows.
	  
	  Finally, we have
	  \begin{equation} \label{eq : InitialTerm}
	    - \sum \limits_{i \in V_1^{1}} \sigma_i^F d(F_i,P^1) - \sum \limits_{i \in V_2^{1}} \sigma_i^F d(F_i,P^2) =
	    - \sum \limits_{i \in V} \sigma_{i}^F d(F_i,P),
	  \end{equation}
	  due to the fact that $V_1^{1},V_2^{1}$ form a partition of $V$ and \eqref{eq : InitPart}.
	  
	  Similar arguments as above show
	  \begin{eqnarray*}
	    \sigma_i^h (\vert h_i + \nabla h_i s^{t} \vert - \vert h_i\vert) -
	    \sum \limits_{\tau = 1}^{t} \lambda_i^{h,\tau} \nabla h_i (s^{\tau} - s^{\tau-1}) & \leq & 
	    (\sigma_i^h - \tilde{\lambda}_i^{h}) (\delta^{t} - 1) \vert h_i \vert, i \in E, \\
	    \sigma_i^g ( (g_i + \nabla g_i s^{t})^+ - (g_i)^+ ) -
	    \sum \limits_{\tau = 1}^{t} \lambda_i^{g,\tau} \nabla g_i (s^{\tau} - s^{\tau-1}) & \leq &
	    (\sigma_i^g - \tilde{\lambda}_i^{g}) (\delta^{t} - 1) ( g_i )^+, i \in I.
	  \end{eqnarray*}
	  Taking this into account and putting together \eqref{eqn : TheLong}, \eqref{eq : MultiplTerm}, \eqref{eq : FinalTerm} and \eqref{eq : InitialTerm}
	  we obtain for $j=0,1$
	  \begin{eqnarray*}
	    r^{t+j}_{1-j} & \leq & - \sum \limits_{\tau = 1}^{t} \frac{1}{2} (s^{\tau} - s^{\tau-1})^T B (s^{\tau} - s^{\tau-1}) \\
	    && - \sum \limits_{i \in V} (\sigma_i^F - \tilde{\lambda}_{i}^F) (1 - \delta^t) d(F_i,P)
	    - \sum \limits_{i \in E} (\sigma_i^h - \tilde{\lambda}_i^{h}) (1 - \delta^{t}) \vert h_i \vert
	    - \sum \limits_{i \in I} (\sigma_i^g - \tilde{\lambda}_i^{g}) (1 - \delta^{t}) ( g_i )^+
	  \end{eqnarray*}
	  and hence \eqref{eqn : BoundMain} and \eqref{eqn : BoundMain2} follow by monotonicity
	  of $\delta$ and \eqref{EqPropSigma}. This completes the proof.
      \end{proof}
    
    \subsubsection{Searching for the next iterate}
    
      We choose the next iterate as a point from the polygonal line connecting the points $s_k^0, \ldots, s_k^{N_k}$.
      Each line segment $[s_k^{t-1},s_k^t]$ corresponds to the convex subproblem solved by Algorithm \ref{AlgSol}
      and hence each line search function $\hat\phi_k^t$ corresponds to the usual $\ell_1$ merit function from
      nonlinear programming. This makes it technically more difficult to prove
      the convergence behavior stated in Proposition \ref{pro : ToZeroAtN} which is also the motivation for the following
      procedure.
      
      First we parametrize the polygonal line connecting the points $s_k^0, \ldots, s_k^{N_k}$
      by its length as a curve $\hat s_k: [0,1] \to \R^n$ in the following way.
      We define $t_k(1) := N_k$, for every $\gamma \in [0,1)$
      we denote by $t_k(\gamma)$ the smallest number $t$ such that $S_k^t > \gamma S_k^{N_k}$ and we set $\alpha_k(1) := 1$,
      \[\alpha_k(\gamma) := \frac{\gamma S_k^{N_k} - S_k^{t_k(\gamma)-1}}{S_k^{t_k(\gamma)} - S_k^{t_k(\gamma)-1}}, \gamma \in [0,1), \]
      where $S_k^0 := 0, S_k^t := \sum_{\tau=1}^{t} \Vert s_k^{\tau} - s_k^{\tau-1} \Vert$ for $t=1, \ldots, N_k$.
      Then we define
      \[ \hat s_k (\gamma) = s_k^{t_k(\gamma)-1} + \alpha_k(\gamma)(s_k^{t_k(\gamma)} - s_k^{t_k(\gamma)-1}).\]
      Note that $\norm{\hat s_k (\gamma)} \leq \gamma S_k^{N_k}$.
      
      In order to simplify the proof of Proposition \ref{pro : ToZeroAtN},
      for $\gamma \in [0,1]$ we further consider the following line search functions
      \begin{equation}
      \begin{array}{c}
	Y_k(\gamma) := \phi_k^{t_k(\gamma)}(\alpha_k(\gamma)), \quad \hat Y_k(\gamma) := \hat\phi_k^{t_k(\gamma)}(\alpha_k(\gamma)), \\
	Z_k(\gamma) := (1 - \alpha_k(\gamma)) \hat\phi_k^{t_k(\gamma)}(0) + \alpha_k(\gamma) \hat\phi_k^{t_k(\gamma)}(1).
      \end{array}
      \end{equation}
      
      Now consider some sequence of positive numbers $\gamma^k_1 = 1, \gamma^k_2, \gamma^k_3, \ldots$ with
      $1 > \bar \gamma \geq \gamma^k_{j+1} / \gamma^k_j \geq \underline \gamma > 0$ for all $j \in \mathbb{N}$.
      Consider the smallest $j$, denoted by $j(k)$ such that for some given constant $\xi \in (0,1)$ one has
      \begin{equation} \label{eq : NextIterCond}
	Y_k(\gamma^k_j) - Y_k(0) \leq \xi \left( Z_k(\gamma_{j}^{k}) - Z_k(0) \right).
      \end{equation}
      Then the new iterate is given by
      \[ x_{k+1} := x_k + \hat s_k(\gamma^k_{j(k)}).\]
      As can be seen from the proof of Lemma \ref{lem : NextIterCons}, this choice ensures a decrease in merit function
      $\Phi$ defined in the next subsection.
      
      The following relations are direct consequences of the properties of $\phi_k^t$ and $\hat\phi_k^t$
      \begin{equation} \label{eq : PropOfNewLSF}
	\vert Y_k(\gamma) - \hat Y_k(\gamma) \vert = o(\gamma S_k^{N_k}), \quad
	\hat Y_k(\gamma) \leq Z_k(\gamma), \quad
	Z_k(\gamma) - Z_k(0) \leq 0.
      \end{equation}
      The last property holds due to Proposition \ref{pro : MainMerit} and
      \begin{equation} \label{eq : Z}
	Z_k(\gamma) - Z_k(0) = (1 - \alpha_k(\gamma)) r_{k,0}^{t_k(\gamma)} + \alpha_k(\gamma) r_{k,1}^{t_k(\gamma)},
      \end{equation}
      which follows from $\alpha_k(0) = 0$, $S_k^{t_k(0)-1} = 0$ and hence $\hat\phi_k^{t_k(0)}(0) = \hat\phi_k^{1}(0)$.
      We recall that $r^t_{k,0}$ and $r^{t}_{k,1}$ are defined by \eqref{eq : DefOfR}.

      \begin{lemma} \label{lem : WellDef}
	  The new iterate $x_{k+1}$ is well defined.
	\end{lemma}
	\begin{proof}
	  In order to show that the new iterate is well defined, we have to prove the existence of some $j$ such that \eqref{eq : NextIterCond}
	  is fulfilled. Note that $S_k^{t_k(0) - 1} = 0$ and $S_k^{t_k(0)} > 0$. There is some $\delta_k > 0$ such that
	  $\vert Y_k(\gamma) - \hat{Y}_k(\gamma) \vert \leq \frac{-(1 - \xi)r^{t_k(0)}_{k,1} \gamma S_k^{N_k}}{S_k^{t_k(0)}}$,
	  whenever $\gamma S_k^{N_k} \leq \delta_k$. Since $\lim_{j \to \infty} \gamma_j^k = 0$, we can choose $j$ sufficiently large
	  to fulfill $\gamma_j^k S_k^{N_k} < \min \{ \delta_k, S_k^{t_k(0)} \}$ and then $t_k(\gamma_j^k) = t_k(0)$ and
	  $\alpha_k(\gamma_j^k) = \gamma_j^k S_k^{N_k} / S_k^{t_k(0)}$, since $S_k^{t_k(0) - 1} = 0$.
	  This yields
	  \begin{equation} \label{eq : MainTrick}
	    Y_k(\gamma_j^k) - \hat{Y}_k(\gamma_j^k) \leq
	    - (1 - \xi) \alpha_k(\gamma_j^k) r^{t_k(\gamma_j^k)}_{k,1}.
	  \end{equation}
	  Then by second property of \eqref{eq : PropOfNewLSF}, \eqref{eq : Z}, taking into account 
	  $r^{t_k(\gamma_j^k)}_{k,0} \leq 0$ by Proposition \ref{pro : MainMerit}
	  and $Y_k(0) = Z_k(0)$ we obtain
	  \begin{eqnarray*}
	    Y_k(\gamma_j^k) - Y_k(0) & \leq &
	    \hat Y_k(\gamma_j^k) - Y_k(0) - (1 - \xi) \alpha_k(\gamma_j^k) r^{t_k(\gamma_j^k)}_{k,1} \\
	    & \leq & \xi (Z_k(\gamma_j^k) - Z_k(0)) + (1 - \xi) \left( Z_k(\gamma_j^k) - Z_k(0) - \alpha_k(\gamma_j^k) r^{t_k(\gamma_j^k)}_{k,1} \right) \\
	    & \leq & \xi (Z_k(\gamma_j^k) - Z_k(0)) + (1 - \xi) (1 - \alpha_k(\gamma_j^k)) r^{t_k(\gamma_j^k)}_{k,0} \leq \xi (Z_k(\gamma_j^k) - Z_k(0)).
	  \end{eqnarray*}
	  Thus \eqref{eq : NextIterCond} is fulfilled for this $j$ and the lemma is proved.
	\end{proof}
	
    \subsection{Convergence of the basic algorithm}
      
      We consider the behavior of the Algorithm \ref{AlgMPCC} when it does not prematurely stop and it generates an infinite
      sequence of iterates
      \[x_k, B_k, (s_k^{t}, \delta_k^{t}), \lambda_k^{t}, (V_{1,k}^{t}, V_{2,k}^{t}), t = 0, \ldots, N_k \textrm{ and }
      \Beta_k, \underline\lambda_k^{N_k}, \overline\lambda_k^{N_k}.\]
      Note that $\delta_k^{N_k} < \zeta$. We discuss the convergence behavior under the following assumption.
      \begin{assumption} \label{ass : AlgMPCC}
	\begin{enumerate}
	  \item There exist constants $C_x, C_s, C_{\lambda}$ such that
		$$\norm{x_k} \leq C_x, \quad S_k^{N_k} \leq C_s, \quad
		\hat \lambda_k^h, \hat \lambda_k^g, \hat \lambda_k^F \leq C_{\lambda}$$
		for all $k$, where $\hat \lambda_k^h := \max_{i \in E} \{ \tilde \lambda_{i,k}^h \}$,
		$\hat \lambda_k^g := \max_{i \in I} \{ \tilde \lambda_{i,k}^g \}$,
		$\hat \lambda_k^F := \max_{i \in V} \{ \tilde \lambda_{i,k}^F \}$.
	  \item There exist constants $\bar C_{B}, \underbar C_{B}$ such that $\underbar{C}_B \leq \lambda(B_k), \norm{B_k} \leq \bar C_B$
		for all $k$, where $\lambda(B_k)$ denotes the smallest eigenvalue of $B_k$.
	\end{enumerate}
      \end{assumption}
      
      For our convergence analysis we need one more merit function
	\[
	  \Phi_k(x) := f(x) + \sum \limits_{i \in E} \sigma_{i,k}^h \vert h_i(x) \vert + \sum \limits_{i \in I} \sigma_{i,k}^g ( g_i(x) )^+
	  + \sum \limits_{i \in V} \sigma_{i,k}^F d(F_i(x),P).
	\]
	
	\begin{lemma} \label{lem : GeneralMerit}
	 For each $k$ and for any $\gamma \in [0,1]$ it holds that
	  \begin{equation} \label{eq: generals}
	    \Phi_k(x_k + \hat s_k(\gamma)) \leq Y_k(\gamma) \quad \textrm{ and } \quad \Phi_k(x_k) = Y_k(0).
	  \end{equation}
	\end{lemma}
	\begin{proof}
	  The first claim follows from the definitions of $\Phi_k$ and $Y_k$ and the estimate
	  \[d(F_i(x_k + s),P^1), d(F_i(x_k + s),P^2) \geq \min\{ d(F_i(x_k + s),P^1), d(F_i(x_k + s),P^2) \} =
	  d(F_i(x_k + s),P),\]
	  which holds by \eqref{eq : DistOfCup}. The second claim follows from \eqref{eq : InitPart}.
	\end{proof}
      
      A simple consequence of the way that we define the penalty parameters in \eqref{eqn : PenalParam} is the following lemma.
      \begin{lemma} \label{lem : Sigmas}
	Under Assumption \ref{ass : AlgMPCC} there exists some $\bar k$ such that for all $k \geq \bar k$ the penalty parameters remain constant,
	$\bar \sigma := \sigma_{k}$ and consequently $\Phi_k(x) = \Phi_{\bar k}(x)$.
      \end{lemma}

      \begin{remark}
	Note that we do not use $\Phi_k$ for calculating the new iterate because its first order approximation is in general not convex
	on the line segments connecting $s_k^{t-1}$ and $s_k^{t}$ due to the involved min operation.
      \end{remark}
      
      \begin{lemma} \label{lem : NextIterCons}
	Assume that Assumption \ref{ass : AlgMPCC} is fulfilled. Then
	\begin{equation} \label{eq : FToZero}
	  \lim_{k \to \infty} Y_k(\gamma_{j(k)}^{k}) - Y_k(0) = 0. 
	\end{equation}
      \end{lemma}
      \begin{proof}
	Take an existed $\bar k$ from Lemma \ref{lem : Sigmas}. Then we have for $k \geq \bar k$
	\[ \Phi_{k+1}(x_{k+1}) = \Phi_{\bar k}(x_{k+1}) = \Phi_{\bar k}(x_{k} + \hat s_k(\gamma^k_{j(k)})) = \Phi_{k}(x_{k} + \hat s_k(\gamma^k_{j(k)}))
	\leq Y_k(\gamma_{j(k)}^{k}) < Y_k(0) = \Phi_k(x_k) \]
	and therefore $\Phi_{k+1}(x_{k+1}) - \Phi_k(x_k) \leq Y_k(\gamma_{j(k)}^{k}) - Y_k(0) < 0$.
	Hence the sequence $\Phi_k(x_k)$ is monotonically decreasing and therefore convergent, because it is bounded below by Assumption \ref{ass : AlgMPCC}.
	Hence
	\[ - \infty < \lim_{k \to \infty} \Phi_k(x_k) - \Phi_{\bar k}(x_{\bar k}) = \sum_{k = \bar k}^{\infty} (\Phi_{k+1}(x_{k+1}) - \Phi_k(x_k))
	\leq \sum_{k = \bar k}^{\infty} (Y_k(\gamma_{j(k)}^{k}) - Y_k(0)) \]
	and the assertion follows.
      \end{proof}
      
      \begin{proposition} \label{pro : ToZeroAtN}
	Assume that Assumption \ref{ass : AlgMPCC} is fulfilled. Then
	\begin{equation} \label{eq : RTtoZero}
	  \lim_{k \to \infty} \hat Y_k(1) - \hat Y_k(0) = 0 
	\end{equation}
	and consequently
	\begin{equation} \label{eq : STtoZero}
	  \lim_{k \to \infty} \norm{s_k^{N_k}} = 0.
	\end{equation}
	\end{proposition}
	\begin{proof}
	We prove \eqref{eq : RTtoZero} by contraposition. Assuming on the contrary that \eqref{eq : RTtoZero} does not hold,
	by taking into account $\hat Y_k(1) - \hat Y_k(0) \leq 0$ by Proposition \ref{pro : MainMerit}, there
	exists a subsequence $K = \{ k_1, k_2, \ldots \}$ such that $\hat Y_k(1) - \hat Y_k(0) \leq \bar r < 0$.
	By passing to a subsequence we can assume that for all $k \in K$ we have $k \geq \bar k $ with $\bar k$ given by Lemma \ref{lem : Sigmas}
	and $N_k = \bar N$, where we have taken into account \eqref{eq : StepsBound}. By passing to a subsequence once more
	we can also assume that
	\[ \lim_{k \setto K \infty} S_k^t = \bar S^t, \lim_{k \setto K \infty} r_{k,1}^t = \bar r_1^t,
	\lim_{k \setto K \infty} r_{k,0}^t = \bar r_0^t, \,\, \forall t \in \{1, \ldots, \bar N\}, \]
	where $r_{k,1}^t$ and $r_{k,0}^t$ are defined by \eqref{eq : DefOfR}.
	Note that $\bar r_1^{\bar N} \leq \bar r < 0$.
	
	Let us first consider the case $\bar{S}^{\bar N} = 0$. 
	There exists $\delta > 0$ such that $\vert Y_k(\gamma) - \hat{Y}_k(\gamma) \vert
	\leq (\xi - 1) \bar{r}_1^{\bar N} \gamma S_k^{\bar N} \, \forall k \in K,$
	whenever $\gamma S_k^{\bar N} \leq \delta$.
	Since $\bar{S}^{\bar N} = 0$ we can assume that
	$S_k^{\bar N} \leq \min \{ \delta, 1/2 \} \, \forall k \in K$. Then
	\[
	  Y_k(1) - Y_k(0) \leq r_{k,1}^{\bar N} + (\xi - 1) \bar{r}_1^{\bar N} S_k^{\bar N}
	  \leq r_{k,1}^{\bar N} + (\xi - 1) r_{k,1}^{\bar N} = \xi r_{k,1}^{\bar N} = \xi (Z_k(1) - Z_k(0)) \leq \frac{\xi \bar{r}_1^{\bar N}}{2} < 0
	\]
	and this implies that for the next iterate we have $j(k) = 1$ and hence $\gamma_{j(k)}^k = 1$, contradicting \eqref{eq : FToZero}.
	
	Now consider the case $\bar{S}^{N} \neq 0$ and let us define the number $\bar \tau := \max \{ t \mv \bar{S}^{t} = 0 \} + 1$.
	Note that Proposition \ref{pro : MainMerit} yields
	\begin{equation} \label{eq : RvsSrelat}
	  r_{k,1}^t, r_{k,0}^{t+1} \leq - \frac{\lambda(B_k)}{2} \sum_{\tau = 1}^{t} \norm{s_k^{\tau} - s_k^{\tau -1}}^2
	  \leq - \frac{\underbar C_{B}}{2} \frac{1}{t} \left( \sum_{\tau = 1}^{t} \norm{s_k^{\tau} - s_k^{\tau -1}} \right)^2 
	  = - \frac{\underbar C_{B}}{2} \frac{1}{t} (S_k^t)^2
	\end{equation}
	and therefore $\tilde r := \max_{t > \bar \tau} \bar r^t < 0$, where $\bar r^t := \max\{ \bar r_0^t, \bar r_1^t \}$.
	By passing to a subsequence we can assume
	that for every $t > \bar \tau$ and every $k \in K$ we have $r_{k,0}^t,r_{k,1}^{t} \leq \frac{\bar r^t}{2}$.
	
	Now assume that for infinitely many $k \in K$ we have $\gamma_{j(k)}^k S_k^{\bar N} \geq S_k^{\bar \tau}$, i.e.
	$t_k(\gamma_{j(k)}^{k}) > \bar \tau$.
	Then we conclude
	\[
	  Y_k(\gamma_{j(k)}^{k}) - Y_k(0) \leq \xi (Z_k(\gamma_{j(k)}^{k}) - Z_k(0)) = 
	  \xi \left( (1 - \alpha_k(\gamma_{j(k)}^{k})) r_{k,0}^{t_k(\gamma_{j(k)}^{k})}
	  + \alpha_k(\gamma_{j(k)}^{k}) r_{k,1}^{t_k(\gamma_{j(k)}^{k})} \right) \leq \frac{\xi \tilde r}{2} < 0
	\]
	contradicting \eqref{eq : FToZero}. Hence for all but finitely many $k \in K$, without
	loss of generality for all $k \in K$, we have $\gamma_{j(k)}^k S_k^{\bar N} < S_k^{\bar \tau}$.
	
	There exists $\delta > 0$ such that
	\begin{equation} \label{eq : MainEstim}
	\vert Y_k(\gamma) - \hat{Y}_k(\gamma) \vert \leq \frac{\vert \bar{r}^{\bar{\tau}} \vert (1 - \xi) \underline \gamma
	\gamma S_k^{\bar N}}{8 S^{\bar \tau}} \, \forall k \in K,
	\end{equation}
	whenever $\gamma S_k^{\bar N} \leq \delta$.
	By eventually choosing $\delta$ smaller we can assume $\delta \leq S^{\bar \tau} / 2$ and by passing to a subsequence
	if necessary we can also assume that for all $k \in K$ we have
	\begin{equation} \label{eq : FirstAs}
	  2 S_k^{\bar \tau-1} / \underline \gamma \leq \delta < S_k^{\bar \tau} \leq 2 S^{\bar \tau}. 
	\end{equation}
	
	Now let for each $k$ the index $\tilde j(k)$ denote the smallest $j$ with $\gamma_j S_k^{\bar N} \leq \delta$.
	It obviously holds that $\gamma_{\tilde j(k)-1}^{k} S_k^{\bar N} > \delta$ and by \eqref{eq : FirstAs} we obtain
	\[S_k^{\bar \tau-1} \leq \underline \gamma \delta \leq \underline \gamma \gamma_{\tilde j(k)-1}^{k} S_k^{\bar N}
	\leq \gamma_{\tilde j(k)}^{k} S_k^{\bar N} \leq \delta < S_k^{\bar \tau}\]
	implying $t_k(\gamma_{\tilde j(k)}^{k}) = \bar \tau$ and
	\[\alpha_k(\gamma_{\tilde j(k)}^{k}) \geq \frac{\underline{\gamma} \delta - S_k^{\bar \tau -1}}
	  {S_k^{\bar \tau} - S_k^{\bar \tau -1}} \geq \frac{\underline{\gamma} \delta}{4 S^{\bar \tau}}\]
	by \eqref{eq : FirstAs}.

	Taking this into account together with \eqref{eq : MainEstim} and $\gamma_{\tilde j(k)}^{k} S_k^{\bar N} \leq \delta$ we conclude
	\[
	  Y_k(\gamma_{\tilde j(k)}^{k}) - \hat{Y}_k(\gamma_{\tilde j(k)}^{k}) \leq
	  \frac{\vert \bar{r}^{\bar{\tau}} \vert (1 - \xi) \underline{\gamma} \gamma_{\tilde j(k)}^{k} S_k^{\bar N}}{8 S^{\bar \tau}}
	  \leq - (1 - \xi) \frac{\underline{\gamma} \delta}{4 S^{\bar \tau}}r_{k,1}^{\bar \tau}
	  \leq - (1 - \xi) \alpha_k(\gamma_{\tilde j(k)}^{k}) r_{k,1}^{t_k(\gamma_{\tilde j(k)}^{k})}.
	\]
	Now we can proceed as in the proof of Lemma \ref{lem : WellDef} to show that $\tilde j(k)$ fulfills \eqref{eq : NextIterCond}.

	However, this yields $\tilde j(k) \geq j(k)$ by definition of $j(k)$ and hence
	$\gamma_{j(k)}^{k} S_k^{\bar N} \geq \gamma_{\tilde j(k)}^{k} S_k^{\bar N} \geq S_k^{\bar \tau-1}$
	showing $t_k(\gamma_{j(k)}^{k}) = t_k(\gamma_{\tilde j(k)}^{k}) =\bar \tau$. But then we also have
	$\alpha_k(\gamma_{j(k)}^{k}) \geq \alpha_k(\gamma_{\tilde j(k)}^{k}) \geq \frac{\underline{\gamma} \delta}{4 \bar S^{\bar \tau}}$ and from
	\eqref{eq : NextIterCond} we obtain
	\[
	  Y_k(\gamma_{j(k)}^{k}) - Y_k(0) \leq \xi (Z_k(\gamma_{j(k)}^{k}) - Z_k(0)) \leq \xi \alpha_k(\gamma_{j(k)}^{k}) r_{k,1}^{t_k(\gamma_{j(k)}^{k})}
	  \leq \frac{\xi \underline{\gamma} \delta \tilde r}{8 \bar S^{\bar \tau}} < 0
	\]
	contradicting \eqref{eq : FToZero} and so \eqref{eq : RTtoZero} is proved.
	Condition \eqref{eq : STtoZero} now follows from \eqref{eq : RTtoZero} because we conclude from \eqref{eq : RvsSrelat} that
	$\hat Y_k(1) - \hat Y_k(0) \leq - \frac{\underbar C_{B}}{2} \frac{1}{N_k} (S_k^{N_k})^2
	\leq - \frac{\underbar C_{B}}{2} \frac{1}{N_k} \norm{s_k^{N_k}}^2$.
      \end{proof}
      
      Now we are ready to state the main result of this section.
      
      \begin{theorem} \label{The : Mstat}
	Let Assumption \ref{ass : AlgMPCC} be fulfilled. Then every limit point of the sequence of iterates $x_k$
	is at least M-stationary for problem \eqref{eq : genproblem}. 
      \end{theorem}
      \begin{proof}
	Let $\bar{x}$ denote a limit point of the sequence $x_k$ and let $K$ denote a subsequence such that
	$\lim_{k \setto K \infty} x_k = \bar x$. Further let $\underline \lambda$ be a limit point of the bounded sequence
	$\underline \lambda_k^{N_k}$ and assume without loss of generality that
	$\lim_{k \setto K \infty} \underline \lambda_k^{N_k} = \underline \lambda$.
	First we show feasibility of $\bar{x}$ for the problem \eqref{eq : genproblem} together with
	\begin{equation} \label{eq : FinalMultInNC}
	  \underline \lambda_i^g \geq 0 = \underline \lambda_i^g g_i(\bar x), i \in I \quad \textrm{ and } \quad
	  (\underline \lambda^{H}, \underline \lambda^{G}) \in N_{P^{\vert V \vert}}(F(\bar x)).
	\end{equation}
	
	Consider $i \in I$. For all $k$ it holds that
	\[0 \geq \left( (1 - \Beta_{i,k}^{g} \delta_k^{N_k}) g_i(x_k) + \nabla g_i(x_k) s_k^{N_k} \right) \perp
	\underline \lambda_{i,k}^{g,N_k} \geq 0.\]
	Since $0 \leq \delta_k^{N_k} \leq \zeta$, $\Beta_{i,k}^{g} \in \{0,1\}$ we have
	$1 \geq (1 - \Beta_{i,k}^{g} \delta_k^{N_k}) \geq 1 - \zeta$ and together with $s_k^{N_k} \to 0$ by Proposition \ref{pro : ToZeroAtN}
	we conclude
	\[ 0 \geq \limsup_{k \setto K \infty} \left( g_i(x_k) + \frac{\nabla g_i(x_k) s_k^{N_k}}{(1 - \Beta_{i,k}^{g} \delta_k^{N_k})} \right)
	= g_i(\bar x),\]
	$\underline \lambda_i^g \geq 0$ and
	\[0 = \lim_{k \setto K \infty} \underline \lambda_{i,k}^{g,N_k} \left( g_i(x_k) + \frac{\nabla g_i(x_k) s_k^{N_k}}
	{(1 - \Beta_{i,k}^{g} \delta_k^{N_k})} \right)
	= \underline \lambda_i^g g_i(\bar x).\]
	Hence $\underline \lambda_i^g \geq 0 = \underline \lambda_i^g g_i(\bar x)$.
	Similar arguments show that for every $i \in E$ we have
	\[ 0 = \lim_{k \setto K \infty} \left( h_i(x_k) + \frac{\nabla h_i(x_k) s_k^{N_k}}{(1 - \delta_k^{N_k})} \right)
	= h_i(\bar x).\]
	
	Finally consider $i \in V$. Taking into account \eqref{eq : DistIfIn}, \eqref{eq : betaeff} and $\delta_k^{N_k} \leq \zeta$ we obtain
	\begin{eqnarray*}
	  d(F_i(x_k),P) & \leq & \norm{ \delta_k^{N_k} (\Beta_{i,k}^H H_i(x_k), - \Beta_{i,k}^G G_i(x_k))^T + \nabla F_i(x_k) s_k^{N_k}}_1 \\
	  & \leq & \zeta d(F_i(x_k),P) + \norm{\nabla F_i(x_k) s_k^{N_k}}_1.
	\end{eqnarray*}
	Hence, $\nabla F_i(x_k) s_k^{N_k} \to 0$ by Proposition \ref{pro : ToZeroAtN} implies
	\[(1-\zeta)d(F_i(\bar x),P) = \lim_{k \setto K \infty} (1-\zeta) d(F_i(x_k),P) \leq
	\lim_{k \setto K \infty} \norm{\nabla F_i(x_k) s_k^{N_k}}_1 = 0,\]
	showing the feasibility of $\bar x$. Moreover, the previous arguments also imply
	\begin{equation} \label{eq : AuxFtoF}
	  \tilde F_i(x_k,s_k^{N_k},\delta_k^{N_k}) := \delta_k^{N_k} (\Beta_{i,k}^H H_i(x_k), - \Beta_{i,k}^G G_i(x_k))^T
	  + F_i(x_k) + \nabla F_i(x_k) s_k^{N_k} \setto K F_i(\bar x).
	\end{equation}
	
	Taking into account \eqref{eq : MstatLNC}, the fact that $\underline \lambda_k^{N_k}$ fulfills M-stationarity conditions
	at $(s_k^{N_k},\delta_k^{N_k})$ for \eqref{eq : deltaprob} yields
	\[(\underline \lambda_{k}^{H,N_k}, \underline \lambda_{k}^{G,N_k}) \in N_{P^{\vert V \vert}}(\tilde F(x_k,s_k^{N_k},\delta_k^{N_k})).\]
	However, this together with $(\underline \lambda_{k}^{H,N_k}, \underline \lambda_{k}^{G,N_k}) \setto K
	(\underline \lambda^{H}, \underline \lambda^{G})$, \eqref{eq : AuxFtoF}, and \eqref{eq : propLimitNC} yield
	$(\underline \lambda^{H}, \underline \lambda^{G}) \in N_{P^{\vert V \vert}}(F(\bar x))$
	and consequently \eqref{eq : FinalMultInNC} follows.
	
	Moreover, by first order optimality condition we have
	\[
	  B_k s_k^{N_k} + \nabla f(x_k)^T + \sum \limits_{i \in E} \underline \lambda_{i,k}^{h,N_k} \nabla h_i(x_k)^T
	  + \sum \limits_{i \in I} \underline \lambda_{i,k}^{g,N_k} \nabla g_i(x_k)^T
	  + \sum \limits_{i \in V} \nabla F_i(x_k)^T \underline \lambda_{i,k}^{F,N_k} = 0
	\]
	for each $k$ and by passing to a limit and by taking into account that $B_ks_k^{N_k} \to 0$
	by Proposition \ref{pro : ToZeroAtN} we obtain
	\[
	  \nabla f(\bar{x})^T + \sum \limits_{i \in E} \underline\lambda_{i}^{h} \nabla h_i(\bar{x})^T
	  + \sum \limits_{i \in I} \underline\lambda_{i}^{g} \nabla g_i(\bar{x})^T
	  + \sum \limits_{i \in V} \nabla F_i(\bar{x})^T \underline \lambda_{i}^{F} = 0.\]
	Hence, invoking \eqref{eq : MstatLNC} again, this together with the feasibility of $\bar x$ and \eqref{eq : FinalMultInNC}
	implies M-stationarity of $\bar x$ and the proof is complete.
      \end{proof}
      
\section{The extended SQP algorithm for MPVC}
      
      In this section we investigate what can be done in order to secure $\mathcal{Q}_M$-stationarity of the limit points.
      First, note that to prove M-stationarity of the limit points in Theorem \ref{The : Mstat} we only used that
      $(\underline \lambda_{k}^{H,N_k}, \underline \lambda_{k}^{G,N_k}) \in N_{P^{\vert V \vert}}(\tilde F(x_k,s_k^{N_k},\delta_k^{N_k}))$,
      i.e. it is sufficient to exploit only the M-stationarity of the solutions of auxiliary problems.
      Further, recalling the comments after Lemma \ref{Lem : QstatQPpiece}, the solution $(s,\delta)$ of $QP(\rho, I^{1}(s,\delta) \cup I^{00}(s,\delta))$
      is M-stationary for the auxiliary problem. Thus, in Algorithm \ref{AlgSol} for solving the auxiliary problem,
      it is sufficient to consider only the last problem of the four problems \eqref{eqn : b1b2Stat},\eqref{eqn : I00EmptyStat}.
      Moreover, definition of limiting normal cone \eqref{eq : LimitNCdef} reveals that, in general, the limiting process abolishes any
      stationarity stronger that M-stationarity, even S-stationarity.
      
      Nevertheless, in practical situations it is likely that some assumption, securing that a stronger stationarity
      will be preserved in the limiting process, may be fulfilled. E.g., let $\bar x$ be a limit point of $x_k$.
      If we assume that for all $k$ sufficiently large it holds that
      $I^{00}(\bar x) = I^{00}(s_k^{N_k},\delta_k^{N_k})$, then $\bar x$ is at least $\mathcal{Q}_M$-stationary for \eqref{eq : genproblem}.
      This follows easily, since now for all $i \in I^{00}(\bar x)$ it holds that $\underline \lambda_{i,k}^{G,N_k} = 0$,
      $\overline \lambda_{i,k}^{H,N_k}, \overline \lambda_{i,k}^{G,N_k} \geq 0$ and consequently
      \[\underline \lambda_i^{G} = \lim_{k \to \infty} \underline \lambda_{i,k}^{G,N_k} = 0, \quad
      \overline \lambda_i^{H} = \lim_{k \to \infty} \overline \lambda_{i,k}^{H,N_k} \geq 0, \quad
      \overline \lambda_i^{G} = \lim_{k \to \infty} \overline \lambda_{i,k}^{G,N_k} \geq 0.\]
      
      This observation suggests that to obtain a stronger stationarity of a limit point,
      the key is to correctly identify the bi-active index set at the limit point
      and it serves as a motivation for the extended version of our SQP method.
      Before we can discuss the extended version, we summarize some preliminary results.

    \subsection{Preliminary results}
      
      Let $a: \R^n \rightarrow \R^p$ and $b: \R^n \rightarrow \R^q$ be continuously differentiable.
      Given a vector $x \in \R^n$ we define the linear problem
      \begin{equation} \label{eq : AuxdProb}
      \begin{array}{lrl}
	LP(x) & \min\limits_{d \in \mathbb{R}^{n}} & \nabla f(x) d \\
	& \textrm{subject to } & \phantom{(b(x))^- +} \nabla a(x) d = 0, \\
	&& (b(x))^- + \nabla b(x) d \leq 0, \\
	&& -1 \leq d \leq 1.
      \end{array}
      \end{equation}
      Note that $d=0$ is always feasible for this problem. Next we define a set $A$ by
      \begin{equation} \label{eq : FeasSetDef}
	A := \{x \in \R^n \mv a(x) = 0, b(x) \leq 0\}.
      \end{equation}
	Let $\bar x \in A$ and recall that the Mangasarian-Fromovitz constraint qualification (MFCQ) holds at $\bar x$
	if the matrix $\nabla a(\bar x)$ has full row rank and there exists a vector $d \in \R^n$ such that
	\[\nabla a(\bar x) d = 0, \quad \nabla b_i(\bar x) d < 0, \, i \in \mathcal I(\bar x) := \{i \in \{1, \ldots, q\} \mv b_i(\bar x) = 0\}.\]
	Moreover, for a matrix $M$ we denote by $\norm{M}_p$ the norm given by
	\begin{equation} \label{eq : MatNormDef}
	  \norm{M}_p := \sup \{ \norm{M u}_p \mv \norm{u}_{\infty} \leq 1\}
	\end{equation}
	and we also omit the index $p$ in case $p = 2$.
          
      \begin{lemma} \label{lem : LP1}
	Let $\bar x \in A$, let assume that MFCQ holds at $\bar x$ and let $\bar d$ denote the solution of $LP(\bar x)$.
	Then for every $\epsilon > 0$ there exists $\delta > 0$ such that if $\norm{x - \bar x} \leq \delta$ then
	\begin{equation}
	  \nabla f(x) d \leq \nabla f(\bar x) \bar d + \epsilon,
	\end{equation}
	where $d$ denotes the solution of $LP(x)$.
      \end{lemma}
      \begin{proof}
	The classical Robinson's result (c.f. \cite[Corollary 1, Theorem 3]{Ro76}), together with MFCQ at $\bar x$,
	yield the existence of $\kappa > 0$ and $\tilde \delta > 0$ such that for every $x$ with $\norm{x - \bar x} \leq \tilde \delta$
	there exists $\hat d$ with $\nabla a(x) \hat d = 0$, $(b(x))^- + \nabla b(x) \hat d \leq 0$ and
	\[
	  \norm{\bar d - \hat d} \leq \kappa \max \{ \norm{ \nabla a(x) \bar d }, \norm{ ( (b(x))^- + \nabla b(x) \bar d )^+ } \}
	  =: \nu.
	\]
	Since $\norm{\hat d}_\infty \leq \norm{\hat d - \bar d + \bar d}_\infty \leq 1 + \nu$,
	by setting $\tilde d := \hat d /(1 + \nu)$ we obtain that $\tilde d$ is feasible for $LP(x)$ and
	\[\norm{\bar d - \tilde d} \leq \frac{1}{1 + \nu} \norm{\bar d - \hat d + \nu \bar d} \leq \frac{ (1 + \sqrt n) \nu}{1 + \nu}
	\leq (1 + \sqrt n) \nu.\]
	
	Thus, taking into account $\nabla a(\bar x) \bar d = 0$, $(b(\bar x))^- + \nabla b(\bar x) \bar d \leq 0$
	and $\norm{\bar d}_{\infty} \leq 1$, we obtain
	\[
	\norm{\bar d - \tilde d} \leq (1 + \sqrt n) \kappa \max \{ \norm{ \nabla a(x) - \nabla a(\bar x) },
	\norm{b(x) - b(\bar x)} + \norm{\nabla b(x) - \nabla b(\bar x)} \}.
	\]
	Hence, given $\epsilon > 0$, by continuity of objective and constraint functions as well as their derivatives at $\bar x$ 
	we can define $\delta \leq \tilde \delta$ such that for all $x$ with $\norm{x - \bar x} \leq \delta$ it holds that
	\[\norm{\nabla f(x) - \nabla f(\bar x)}_{1}, \ \norm{\nabla f(x)} \norm{\bar d - \tilde d} \ \leq \ \epsilon / 2.\]
	Consequently, we obtain
	\[
	  \nabla f (x) \tilde d \leq \norm{\nabla f(x)} \norm{\tilde d - \bar d} +
	\norm{\nabla f(x) - \nabla f(\bar x)}_{1} \norm{\bar d}_{\infty} + \nabla f(\bar x) \bar d \leq \nabla f(\bar x) \bar d + \epsilon
	\]
	and since $\nabla f (x) d \leq \nabla f (x) \tilde d$ by feasibility of $\tilde d$ for $LP(x)$, the claim is proved.
      \end{proof}

      \begin{lemma} \label{lem : LP2}      
        Let $\nu \in (0,1)$ be a given constant and
        for a vector of positive parameters $\omega = (\omega^{\mathcal E},\omega^{\mathcal I})$ let us define the following function
      \begin{equation} \label{eq : varphiDef}
	\varphi(x) := f(x) + \sum_{i \in \{1,\ldots,p\}} \omega_i^{\mathcal E} \vert a_i(x) \vert
	+ \sum_{i \in \{1,\ldots,q\}} \omega_i^{\mathcal I} (b_i(x))^+.      
      \end{equation}
	Further assume that there exist $\epsilon > 0$ and a compact set $C$ such that for all $x \in C$
        it holds that $\nabla f(x) d \leq - \epsilon$, where $d$ denotes the solution of $LP(x)$.
        Then there exists $\tilde \alpha > 0$ such that
	\begin{equation} \label{eq : VarphiAlpha}
	  \varphi(x + \alpha d) - \varphi(x) \leq \nu \alpha \nabla f(x) d
	\end{equation}
	holds for all $x \in C$ and every $\alpha \in [0,\tilde \alpha]$.
      \end{lemma}
      \begin{proof}
	Definition of $\varphi$, together with $u^+-v^+ \leq (u - v^+)^+$ for $u,v \in \R$, yield
	\begin{equation} \label{eq : VarPhiBound}
	  \varphi(x + \alpha d) - \varphi(x) \leq f(x + \alpha d) - f(x) +
	  \norm{\omega}_{\infty} ( \norm{a(x + \alpha d) - a(x)}_1 + \norm{(b(x + \alpha d) - (b(x))^+)^+}_1).
	\end{equation}
	By uniform continuity of the derivatives of constraint functions and objective function on compact sets,
	it follows that there exists $\tilde \alpha > 0$ such that for all $x \in C$ and every $h$ with $\norm{h}_{\infty} \leq \tilde \alpha$ we have
	\begin{equation} \label{eq : ContDeriv}
	  \norm{\nabla f(x+h) - \nabla f(x)}_{1}, \, \, \norm{\omega}_{\infty} ( \norm{\nabla a(x + h) - \nabla a(x)}_1 +
	  \norm{\nabla b(x + h) - \nabla b(x)}_1 ) \leq \frac{1-\nu}{2} \epsilon.
	\end{equation}
	
	Hence, for all $x \in C$ and every $\alpha \in [0,\tilde \alpha]$ we obtain
	\begin{eqnarray*}
	  f(x + \alpha d) - f(x) & = & \nu \alpha \nabla f(x) d + (1 - \nu) \alpha \nabla f(x) d +
	  \int_{0}^{1} (\nabla f(x + t \alpha d) - \nabla f(x)) \alpha d \mathrm{d} t \\
	  & \leq & \nu \alpha \nabla f(x) d - (1 - \nu) \alpha \epsilon + \frac{1-\nu}{2} \alpha \epsilon =
	  \nu \alpha \nabla f(x) d - \frac{1-\nu}{2} \alpha \epsilon.
	\end{eqnarray*}
	On the other hand, taking into account $\nabla a(x) d = 0$, $\norm{d}_{\infty} \leq 1$, \eqref{eq : ContDeriv} and
	\[(b(x))^- + \alpha \nabla b(x) d = (1- \alpha)(b(x))^- + \alpha ((b(x))^- + \nabla b(x) d) \leq 0\]
	we similarly obtain for all $x \in C$ and every $\alpha \in [0,\tilde \alpha]$
	{\setlength\arraycolsep{2pt}
	\begin{eqnarray*}
	  \lefteqn{\norm{\omega}_{\infty} ( \norm{a(x + \alpha d) - a(x)}_1 + \norm{(b(x + \alpha d) - (b(x))^+)^+}_1)} \\
	  & \leq & \norm{\omega}_{\infty} \Big( \norm{\smallint_{0}^{1} (\nabla a(x + t \alpha d) - \nabla a(x)) \alpha d \mathrm{d} t}_1
	  + \norm{\smallint_{0}^{1} (\nabla b(x + t \alpha d) - \nabla b(x)) \alpha d \mathrm{d} t}_1 \Big)
	  \leq \frac{1-\nu}{2} \alpha \epsilon.
	\end{eqnarray*}}
	Consequently, \eqref{eq : VarphiAlpha} follows from \eqref{eq : VarPhiBound} and the proof is complete.
      \end{proof}
      
      \subsection{The extended version of Algorithm \ref{AlgMPCC}}

      For every vector $x \in \R^n$ and every partition $(W_1 , W_2) \in \mathcal{P}(V)$
      we define the linear problem
      \begin{equation} \label{eq : dProb}
      \begin{array}{lrll}
	LP(x,W_1) & \min\limits_{d \in \mathbb{R}^{n}} & \nabla f(x) d & \\
	& \textrm{subject to } & \phantom{(g_i(x))^- + } \nabla h_i(x) d = 0 & i \in E, \\
	&& (g_i(x))^- + \nabla g_i(x) d \leq 0 & i \in I, \\
	&& \phantom{(F_i(x))^- + } \nabla F_i(x) d \in P^1 & i \in W_1, \\
	&& (F_i(x))^- + \nabla F_i(x) d \in P^2 & i \in W_2, \\
	&& -1 \leq d \leq 1. &
      \end{array}
      \end{equation}
      Note that $d=0$ is always feasible for this problem and that the problem $LP(x,W_1)$ coincides with the problem $LP(x)$ with $a,b$ given by
      \begin{equation} \label{eq : abDef}
	a := (h_i(x), i \in E, -H_i(x), i \in W_1)^T, \, b := (g_i(x), i \in I, -H_i(x), i \in W_2, G_i(x), i \in W_2)^T.
      \end{equation}      
      
      The following proposition provides the motivation for introducing the problem $LP(x,W_1)$.

      \begin{proposition} \label{Pro : SolLP}
	Let $\bar x$ be feasible for \eqref{eq : genproblem}. Then $\bar x$ is $\mathcal{Q}$-stationary with respect to
		$(\beta^1,\beta^2) \in \mathcal{P}(I^{00}(\bar x))$ if and only if the solutions $\bar d^1$ and $\bar d^2$ of the problems
		$LP(\bar x,I^{0+}(\bar x) \cup \beta^1)$ and $LP(\bar x, I^{0+}(\bar x) \cup \beta^2)$ fulfill
	       \begin{equation} \label{eq : LinSolZero}
		\min \{ \nabla f (\bar x) \bar d^1 , \nabla f (\bar x) \bar d^2 \} = 0.
	       \end{equation}
      \end{proposition}
      \begin{proof}
	Feasibility of $d=0$ for $LP(\bar x,I^{0+}(\bar x) \cup \beta^1)$ and $LP(\bar x,I^{0+}(\bar x) \cup \beta^2)$
	implies
	\[\min \{ \nabla f (\bar x) \bar d^1 , \nabla f (\bar x) \bar d^2 \} \leq 0.\]
	Denote by $\tilde d^1$ and $\tilde d^2$ the solutions
	of $LP(\bar x,I^{0+}(\bar x) \cup \beta^1)$ and $LP(\bar x,I^{0+}(\bar x) \cup \beta^2)$
	without the constraint $-1 \leq d \leq 1$, and denote these problems by
	$\tilde{LP}^1$ and $\tilde{LP}^2$. Clearly, we have
	\[\min \{ \nabla f (\bar x) \tilde d^1 , \nabla f (\bar x) \tilde d^2 \} \leq
	\min \{ \nabla f (\bar x) \bar d^1 , \nabla f (\bar x) \bar d^2 \}.\]
	
	The dual problem of $\tilde{LP}^j$ for $j=1,2$ is given by
	\begin{equation} \label{eq : DualdProb}
	\begin{array}{rl}
	\max\limits_{\lambda \in \R^m}
	& - \sum_{i \in I} \lambda_i^g (g_i(\bar x))^- - \sum_{i \in W_2^j} \left( \lambda_i^H (-H_i(\bar x))^- + \lambda_i^G (G_i(\bar x))^- \right) \\
	\textrm{subject to } & \eqref{eq : StatEq} \,\, \textrm{ and } \,\,
	\lambda_i^g \geq 0, i \in I, \lambda_i^H, \lambda_i^G \geq 0, i \in W_2^j, \lambda_i^G = 0, i \in W_1^j,
	\end{array}
	\end{equation}
	where $\lambda = (\lambda^h,\lambda^g,\lambda^H,\lambda^G)$, $m = \vert E \vert + \vert I \vert + 2 \vert V \vert$,
	$W_1^j := I^{0+}(\bar x) \cup \beta^j$, $W_2^j := V \setminus W_1^j$.
	
	Assume first that $\bar x$ is $\mathcal{Q}$-stationary with respect to $(\beta^1,\beta^2) \in \mathcal{P}(I^{00}(\bar x))$.
	Then the multipliers $\overline\lambda$, $\underline\lambda$ from definition of $\mathcal{Q}$-stationarity are feasible
	for dual problems of $\tilde{LP}^1$ and $\tilde{LP}^2$, respectively, both with the objective value equal to zero.
	Hence, duality theory of linear programming yields that $\min \{ \nabla f (\bar x) \tilde d^1 , \nabla f (\bar x) \tilde d^2 \} \geq 0$
	and consequently \eqref{eq : LinSolZero} follows.
	
	On the other hand, if \eqref{eq : LinSolZero} is fulfilled, is follows that
	$\min \{ \nabla f (\bar x) \tilde d^1 , \nabla f (\bar x) \tilde d^2 \} = 0$
	as well. Thus, $d=0$ is an optimal solution for $\tilde{LP}^1$ and $\tilde{LP}^2$
	and duality theory of linear programming yields that the solutions $\lambda^1$ and $\lambda^2$
	of the dual problems exist and their objective values are both zero.
	However, this implies that for $j=1,2$ we have
	\[\lambda_i^{g,j} g_i(\bar x) = 0, i \in I, \lambda_i^{H,j} H_i(\bar x) = 0 , \lambda_i^{G,j} G_i(\bar x) = 0, i \in V\]
	and consequently $\lambda^1$ fulfills the conditions of $\overline\lambda$ and $\lambda^2$
	fulfills the conditions of $\underline\lambda$,
	showing that $\bar x$ is indeed $\mathcal{Q}$-stationary with respect to $(\beta^1,\beta^2)$.
      \end{proof}
      
      Now for each $k$ consider two partitions $(W_{1,k}^1,W_{2,k}^1), (W_{1,k}^2,W_{2,k}^2) \in \mathcal P(V)$
      and let $d_k^1$ and $d_k^2$ denote the solutions of $LP(x_k, W_{1,k}^1)$ and $LP(x_k, W_{1,k}^2)$.
      Choose $d_k \in \{d_k^1, d_k^2\}$ such that
	\begin{equation} \label{eq : d_kDef}
	  \nabla f (x_k) d_k = \min_{d \in \{d_k^1, d_k^2\}} \nabla f (x_k) d
	\end{equation}
      and let $(W_{1,k},W_{2,k}) \in \{(W_{1,k}^1,W_{2,k}^1), (W_{1,k}^2,W_{2,k}^2)\}$ denote the corresponding partition.
      Next, we define the function $\varphi_k$ in the following way
      \begin{equation} \label{eq : varphikDef}
	\varphi_k(x) := f(x) + \sum \limits_{i \in E} \sigma_{i,k}^h \vert h_i(x) \vert +
	\sum \limits_{i \in I} \sigma_{i,k}^g ( g_i(x) )^+
	+ \sum \limits_{i \in W_{1,k}} \sigma_{i,k}^F d(F_i(x),P^1)
	+ \sum \limits_{i \in W_{2,k}} \sigma_{i,k}^F d(F_i(x),P^2).      
      \end{equation}
      Note that the function $\varphi_k$ coincides with $\varphi$ for $a,b$ given by \eqref{eq : abDef} with $(W_1,W_2) := (W_{1,k},W_{2,k})$
      and $\omega = (\omega^{\mathcal E}, \omega^{\mathcal I})$ given by
	\[
	  \omega^{\mathcal E} := (\sigma^h_{i,k}, i \in E, \sigma^F_{i,k}, i \in W_{1,k}), \qquad
	  \omega^{\mathcal I} := (\sigma^g_{i,k}, i \in I, \sigma^F_{i,k}, i \in W_{2,k}, \sigma^F_{i,k}, i \in W_{2,k}).
	\]
      
      \begin{proposition} \label{Pro : MeritLP}
	For all $x \in \R^n$ it holds that
	\begin{equation} \label{eq : PhiVarphi}
	 0 \leq \varphi_k(x) - \Phi_k(x) \leq \norm{\sigma_k^F}_{\infty} \vert V \vert
	 \max \{ \max_{i \in W_{1,k}} d(F_i(x),P^1), \max_{i \in W_{2,k}} d(F_i(x),P^2) \}.
	\end{equation}
      \end{proposition}
      \begin{proof}
	Non-negativity of the distance function, together with \eqref{eq : DistOfCup} yield for every $i \in V, j = 1,2$
	\[0 \leq d(F_i(x),P^j) - d(F_i(x),P) \leq d(F_i(x),P^j). \]
	Hence \eqref{eq : PhiVarphi} now follows from
	\[
	 \sum_{j = 1,2} \, \, \sum_{i \in W_{j,k}} \sigma_{i,k}^F d(F_i(x),P^j)
	 \leq \norm{\sigma_k^F}_{\infty} \vert V \vert \max_{j=1,2} \,\, \max_{i \in W_{j,k}} d(F_i(x),P^j).
	\]
      \end{proof}
         
      An outline of the extended algorithm is as follows.
      
      \begin{algorithm}[Solving the MPVC*] \label{AlgMPCC*} \rm \mbox{}
  
	\Itl1{1:} Initialization:
	\Itl2{}   Select a starting point $x_0 \in \R^n$ together with a positive definite $n \times n$ matrix $B_0$,
	\Itl3{}   a parameter $\rho_0 > 0$ and constants $\zeta \in (0,1)$, $\bar\rho > 1$ and $\mu \in (0,1)$.
	\Itl2{}   Select positive penalty parameters $\sigma_{-1} = (\sigma^h_{-1}, \sigma^g_{-1}, \sigma^F_{-1})$.
	\Itl2{}   Set the iteration counter $k := 0$.
	\Itl1{2:} Correction of the iterate:
	\Itl2{}   Set the corrected iterate by $\tilde x_{k} := x_k$.
	\Itl2{}   Take some $(W_{1,k}^1,W_{2,k}^1), (W_{1,k}^2,W_{2,k}^2) \in \mathcal P(V)$, compute $d_k^1$ and $d_k^2$
	\Itl3{}   as solutions of $LP(x_k, W_{1,k}^1)$ and $LP(x_k, W_{1,k}^2)$ and let $d_k$ be given by \eqref{eq : d_kDef}.
	\Itl2{}   Consider a sequence of numbers $\alpha_k^{(1)} = 1, \alpha_k^{(2)}, \alpha_k^{(3)}, \ldots$ with
		  $1 > \bar \alpha \geq \alpha_k^{(j+1)} / \alpha_k^{(j)} \geq \underline \alpha > 0$.
	\Itl2{}   If $\nabla f (x_k) d_k < 0$, denote by $j(k)$ the smallest $j$ fulfilling either
		  \begin{eqnarray} \label{eqn : NextIterCond}
		    \Phi_k(x_k + \alpha_k^{(j)} d_k) - \Phi_k(x_k) & \leq & \mu \alpha_k^{(j)} \nabla f (x_k) d_k, \\ \label{eqn : NextIterCond2}
		    \textrm{or } \qquad \alpha_k^{(j)} & \leq & \frac{\Phi_k(x_k) - \varphi_k(x_k)}{\mu \nabla f (x_k) d_k}.
		  \end{eqnarray}
	\Itl3{}   If $j(k)$ fulfills \eqref{eqn : NextIterCond}, set $\tilde x_{k} := x_k + \alpha_k^{j(k)} d_k$.
	\Itl1{3:} Solve the Auxiliary problem:
	\Itl2{}   Run Algorithm \ref{AlgSol} with data $\zeta, \bar\rho, \rho:= \rho_k, B:=B_k, \nabla f := \nabla f (\tilde x_k),$
	\Itl3{}   $h_i := h_i(\tilde x_k), \nabla h_i := \nabla h_i (\tilde x_k), i \in E,$ etc.
	\Itl2{}   If the Algorithm \ref{AlgSol} stops because of degeneracy,
	\Itl3{}   stop the Algorithm \ref{AlgMPCC*} with an error message.
	\Itl2{}   If the final iterate $s^N$ is zero, stop the Algorithm \ref{AlgMPCC*} and return $\tilde x_k$ as a solution.
	\Itl1{4:} Next iterate:
	\Itl2{}   Compute new penalty parameters $\sigma_{k}$.
	\Itl2{}   Set $x_{k+1} := \tilde x_k + s_k$ where $s_k$ is a point on the polygonal line connecting the points
	\Itl3{}   $s^0,s^1, \ldots, s^N$ such that an appropriate merit function depending on $\sigma_{k}$ is decreased.
	\Itl2{}   Set $\rho_{k+1} := \rho$, the final value of $\rho$ in Algorithm \ref{AlgSol}.
	\Itl2{}   Update $B_{k}$ to get positive definite matrix $B_{k+1}$.
	\Itl2{}   Set $k := k+1$ and go to step 2.
      \end{algorithm}
      
      Naturally, Remark \ref{rem : StoppingCriteria} regarding the stopping criteria for Algorithm \ref{AlgMPCC} aplies to this algorithm as well.
      
      \begin{lemma} \label{lem : WellDefNI}
	Index $j(k)$ is well defined.
      \end{lemma}
      \begin{proof}
	In order to show that $j(k)$ is well defined, we have to prove the existence of some $j$ such that either \eqref{eqn : NextIterCond}
	or \eqref{eqn : NextIterCond2} is fulfilled. By \eqref{eq : PhiVarphi} we know that $\Phi_k(x_k) - \varphi_k(x_k) \leq 0$.
	In case $\Phi_k(x_k) - \varphi_k(x_k) < 0$ every $j$ sufficiently large clearly fulfills \eqref{eqn : NextIterCond2}. On the
	other hand, if $\Phi_k(x_k) - \varphi_k(x_k) = 0$, taking into account \eqref{eq : PhiVarphi} we obtain
	\[\Phi_k(x_k + \alpha d_k) - \Phi_k(x_k) \leq \varphi_k(x_k + \alpha d_k) - \varphi_k(x_k).\]
	However, Lemma \ref{lem : LP2} for $\nu := \mu$ and $C:= \{x_k\}$ yields that if $\nabla f (x_k) d_k < 0$
	then there exists some $\tilde \alpha$ such that
	\[\varphi_k(x_k + \alpha d_k) - \varphi_k(x_k) \leq \mu \alpha \nabla f (x_k) d_k\]
	holds for all $\alpha \in [0,\tilde \alpha]$ and thus \eqref{eqn : NextIterCond} is fulfilled for every $j$ sufficiently large.
	This finishes the proof.
      \end{proof}
      
      \subsection{Convergence of the extended algorithm}

	We consider the behavior of the Algorithm \ref{AlgMPCC*} when it does not prematurely stop and it generates an infinite
	sequence of iterates
	\[x_k, B_k, \Beta_k, \underline\lambda_k^{N_k}, \overline\lambda_k^{N_k},
	(s_k^{t}, \delta_k^{t}), \lambda_k^{t}, (V_{1,k}^{t}, V_{2,k}^{t}), \,\,
	\textrm{ and } \,\, \tilde x_k, d_k^1, d_k^2, (W_{1,k}^1,W_{2,k}^1), (W_{1,k}^2,W_{2,k}^2).\]
	We discuss the convergence behavior under the following additional assumption.
	
	\begin{assumption} \label{ass : AlgMPVC*}
	Let $\bar x$ be a limit point of the sequence of iterates $x_k$.
	\begin{enumerate}
	  \item Mangasarian-Fromovitz constraint qualification (MFCQ) holds at $\bar x$ for
		constraints $x \in A$, where $A$ is given by \eqref{eq : FeasSetDef} and $a,b$ are given by \eqref{eq : abDef}
		with $(W_1,W_2) := (I^{0+}(\bar x),V \setminus I^{0+}(\bar x))$ or
		$(W_1,W_2) := (I^{0+}(\bar x) \cup I^{00}(\bar x),V \setminus (I^{0+}(\bar x) \cup I^{00}(\bar x)))$.
	  \item There exists a subsequence $K(\bar x)$ such that
		$\lim_{k \setto{K(\bar x)} \infty} x_k = \bar x$ and
		\[W_{1,k}^1 = I^{0+}(\bar x), \,\, W_{1,k}^2 = I^{0+}(\bar x) \cup I^{00}(\bar x) \textrm{ for all } k \in K(\bar x).\]
	\end{enumerate}
	\end{assumption}
	
	Note that the Next iterate step from Algorithm \ref{AlgMPCC*} remains almost unchanged compared to the Next iterate step from Algorithm \ref{AlgMPCC},
	we just consider the point $\tilde x_k$ instead of $x_k$. Consequently, most of the results from subsections 4.1 and 4.2 remain valid,
	possibly after replacing $x_k$ by $\tilde x_k$ where needed, e.g. in Lemma \ref{lem : GeneralMerit}. The only exception is
	the proof of Lemma \ref{lem : NextIterCons}, where we have to show that the sequence $\Phi_k(x_k)$ is monotonically decreasing.
	This follows now from \eqref{eqn : NextIterCond} and hence Lemma \ref{lem : NextIterCons} remains valid as well.
	
	We state now the main result of this section.
      
      \begin{theorem} \label{The : Betastat}
	Let Assumption \ref{ass : AlgMPCC} and Assumption \ref{ass : AlgMPVC*} be fulfilled.
	Then every limit point of the sequence of iterates $x_k$
	is at least $\mathcal Q_M$-stationary for problem \eqref{eq : genproblem}.
      \end{theorem}
      \begin{proof}
	Let $\bar{x}$ denote a limit point of the sequence $x_k$ and let $K(\bar x)$ denote a subsequence from Assumption \ref{ass : AlgMPVC*} (2.).
	Since
	\[\norm{x_{k}-\tilde x_{k-1}} \leq S_{k-1}^{N_{k-1}} \to 0 \]
	we conclude that $\lim_{k \setto{K(\bar x)} \infty} \tilde x_{k-1} = \bar x$ and by applying Theorem \ref{The : Mstat} to sequence $\tilde x_{k-1}$
	we obtain the feasibility of $\bar x$ for problem \eqref{eq : genproblem}.
	
	Next we consider $\bar d^1,\bar d^2$ as in Proposition \ref{Pro : SolLP} with $\beta^1 := \emptyset$
	and without loss of generality we only consider $k \in K(\bar x), k \geq \bar k$, where $\bar k$ is given by Lemma \ref{lem : Sigmas}.
	We show by contraposition that the case $\min \{ \nabla f (\bar x) \bar d^1 , \nabla f (\bar x) \bar d^2 \} < 0$ can not occur.
	Let us assume on the contrary that, say $\nabla f (\bar x) \bar d^1 < 0$.
	Assumption \ref{ass : AlgMPVC*} (2.) yields that $W_{1,k}^1 = I^{0+}(\bar x)$ and
	feasibility of $\bar x$ for \eqref{eq : genproblem} together with $I^{0+}(\bar x) \subset W_{1,k}^1 \subset I^{0}(\bar x)$ imply
	$\bar x \in A$ for $A$ given by \eqref{eq : FeasSetDef} and $a,b$ given by \eqref{eq : abDef} with $(W_1,W_2) := (W_{1,k}^1,W_{2,k}^1)$.
	Taking into account Assumption \ref{ass : AlgMPVC*} (1.), Lemma \ref{lem : LP1} then yields that for
	$\epsilon := - \nabla f(\bar x) \bar d^1 /2 > 0$ there exists $\delta$ such that for all $\norm{x_k - \bar x} \leq \delta$
	we have $\nabla f(x_k) d_k \leq \nabla f(x_k) d_k^1 \leq \nabla f(\bar x) \bar d^1 /2 = - \epsilon$, with
	$d_k$ given by \eqref{eq : d_kDef}.
	
	Next, we choose $\hat k$ to be such that for $k \geq \hat k$ it holds that $\norm{x_k - \bar x} \leq \delta$
	and we set $\nu := (1 + \mu)/2$, $C:= \{x \mv \norm{x - \bar x} \leq \delta\}$. From Lemma \ref{lem : LP2} we obtain that
	\begin{equation} \label{eq : varphiInter}
	  \varphi_k(x_k + \alpha d_k) - \varphi_k(x_k) \leq \frac{1+ \mu}{2} \alpha \nabla f(x_k) d_k
	\end{equation}
	holds for all $\alpha \in [0,\tilde \alpha]$.
	Moreover, by choosing $\hat k$ larger if necessary we can assume that for all $i \in V$ we have
	\begin{equation} \label{eq : FCloseToBarx}
	  \norm{F_i(x_k) - F_i(\bar x)}_1 \leq - \min \left\{\frac{1 - \mu}{2} , \mu \right\}
	  \frac{ \underline \alpha \tilde \alpha \nabla f(x_k) d_k}{\norm{\sigma_k^F}_{\infty} \vert V \vert}.
	\end{equation}
	For the partition $(W_{1,k},W_{2,k}) \in \{(W_{1,k}^1,W_{2,k}^1), (W_{1,k}^2,W_{2,k}^2)\}$ corresponding to $d_k$ it holds that
	$I^{0+}(\bar x) \subset W_{1,k} \subset I^{0}(\bar x)$ and this, together with the feasibility of $\bar x$ for \eqref{eq : genproblem},
	imply $F_i(\bar x) \in P^j, i \in W_{j,k}$ for $j=1,2$. Therefore, taking into account \eqref{eq : DistIfIn}, we obtain
	\[\max \{ \max_{i \in W_{1,k}} d(F_i(x_k),P^1), \max_{i \in W_{2,k}} d(F_i(x_k),P^2) \} \leq \max_{i \in V} \norm{F_i(x_k) - F_i(\bar x)}_1.\]
	Consequently, \eqref{eq : PhiVarphi} and \eqref{eq : FCloseToBarx} yield for all $\alpha > \underline \alpha \tilde \alpha$
	\[\varphi(x_k) - \Phi_k(x_k) < - \min \left\{\frac{1 - \mu}{2} , \mu \right\} \alpha \nabla f (x_k) d_k.\]
	Thus, from \eqref{eq : varphiInter} and \eqref{eq : PhiVarphi} we obtain for all $\alpha \in (\underline \alpha \tilde \alpha,\tilde \alpha]$
	\begin{eqnarray*}
	  \Phi_k(x_k + \alpha d_k) - \Phi_k(x_k) & \leq & \varphi(x_k + \alpha d_k) - \varphi(x_k) + \varphi(x_k) - \Phi_k(x_k)
	  \leq \mu \alpha \nabla f (x_k) d_k \\
	  \textrm{and } \qquad \Phi_k(x_k) - \varphi(x_k) & > & \mu \alpha \nabla f (x_k) d_k.
	\end{eqnarray*}
	
	Now consider $j$ with $\alpha_k^{(j-1)} > \tilde \alpha \geq \alpha_k^{(j)}$.
	We see that $\alpha_k^{(j)} \in (\underline \alpha \tilde \alpha,\tilde \alpha]$, since
	$\alpha_k^{(j)} \geq \underline \alpha \alpha_k^{(j-1)} > \underline \alpha \tilde \alpha$
	and consequently $j$ fulfills \eqref{eqn : NextIterCond} and violates \eqref{eqn : NextIterCond2}.
	However, then we obtain for all $k \geq \hat k$
	\[ \Phi_k(x_{k+1}) - \Phi_k(x_{k}) \leq \mu \alpha_k^{(j(k))} \nabla f (x_k) d_k =
	\mu \underline \alpha \tilde \alpha \nabla f(\bar x) \bar d /2 < 0,\]
	a contradiction.
	
	Hence it follows that the solutions $\bar d^1,\bar d^2$ fulfill
	$\min \{ \nabla f (\bar x) \bar d^1 , \nabla f (\bar x) \bar d^2 \} = 0$ and by
	Proposition \ref{Pro : SolLP} we conclude that $\bar x$
	is $\mathcal Q$-stationary with respect to $(\emptyset,I^{00}(\bar x))$ and
	consequently also $\mathcal Q_M$-stationary for problem \eqref{eq : genproblem}.
      \end{proof}
      
      Finally, we discuss how to choose the partitions $(W_{1,k}^1,W_{2,k}^1)$ and $(W_{1,k}^2,W_{2,k}^2)$ such that Assumption \ref{ass : AlgMPVC*} (2.)
      will be fulfilled. Let us consider a sequence of nonnegative numbers $\epsilon_k$ such that for every limit point $\bar x$
      with $\lim_{k \setto K \infty} x_k = \bar x$ it holds that
      \begin{equation} \label{eq : EpsK}
	\lim_{k \setto K \infty} \frac{\epsilon_k}{\norm{x_{k}- \bar x}_{\infty}} \to \infty
      \end{equation}
      and let us define
      \begin{eqnarray*}
      \tilde I^{0+}_k & := & \{ i \in V \mv \vert H_i(x_k) \vert \leq \epsilon_k < G_i(x_k) \}, \\
      \tilde I^{00}_k & := & \{ i \in V \mv \vert H_i(x_k) \vert \leq \epsilon_k \geq \vert G_i(x_k) \vert \}, \\
      \tilde I^{0-}_k & := & \{ i \in V \mv \vert H_i(x_k) \vert \leq \epsilon_k < -G_i(x_k) \}, \\
      \tilde I^{+0}_k & := & \{ i \in V \mv H_i(x_k) > \epsilon_k \geq \vert G_i(x_k) \vert \}, \\
      \tilde I^{+-}_k & := & \{ i \in V \mv H_i(x_k) > \epsilon_k < -G_i(x_k) \}.      
      \end{eqnarray*}
      
      \begin{proposition}
	For $W_{1,k}^1$ and $W_{1,k}^2$ defined by $W_{1,k}^1 := \tilde I^{0+}_k$ and $W_{1,k}^1 := \tilde I^{0+}_k \cup \tilde I^{00}_k$
	the Assumption \ref{ass : AlgMPVC*} (2.) is fulfilled.
      \end{proposition}
      \begin{proof}
	Let $\bar x$ be a limit point of the sequence $x_k$ such that $\lim_{k \setto{K} \infty} x_k = \bar x$.
	Recall that $\mathcal{F}$ is given by \eqref{eqn : mathcFdef} and let us set
	$L := \max_{\norm{x - \bar x}_{\infty} \leq 1} \norm{\nabla \mathcal{F}(x)}_{\infty}$,
	where $\norm{\nabla \mathcal{F}(x)}_{\infty}$ is given by \eqref{eq : MatNormDef}.
	Further, taking into account \eqref{eq : EpsK}, consider $\hat k$ such that for all $k \geq \hat k$ it holds that
	$\norm{x_k - \bar x}_{\infty} \leq \min \left\{ \epsilon_k / L,1 \right\}$.
	Hence, for all $k \in K$ with $k \geq \hat k$ we conclude
	\begin{equation} \label{eq : FxktoBarx}
	  \norm{ \mathcal{F}(x_k) - \mathcal{F}(\bar x)}_{\infty} \leq
	  \int_0^1 \norm{\nabla \mathcal{F}(\bar x + t(x_k - \bar x))}_{\infty}\norm{x_k - \bar x}_{\infty} dt \leq \epsilon_k.
	\end{equation}
	Now consider $i \in I^{0+}(\bar x)$, i.e. $H_i(\bar x) = 0 < G_i(\bar x)$. By choosing $\hat k$ larger if necessary we can assume that for all
	$k \geq \hat k$ it holds that $\epsilon_k < G_i(\bar x)/2$ and consequently, taking into account \eqref{eq : FxktoBarx},
	for all $k \in \{k \in K \mv k \geq \hat k\}$ we have
	\[
	  \vert H_i(x_k) \vert = \vert H_i(x_k) - H_i(\bar x) \vert \leq \epsilon_k < G_i(\bar x) - \epsilon_k \leq G_i(x_k),
	\]
	showing $i \in \tilde I^{0+}_k$. By similar argumentation and by increasing $\hat k$ if necessary we obtain that
	for all $k \in \{k \in K \mv k \geq \hat k\} =: K(\bar x)$ it holds that
	\begin{equation} \label{eq : IndexSetSub}
	  I^{0+}(\bar x) \subset \tilde I^{0+}_k, \,\, I^{00}(\bar x) \subset \tilde I^{00}_k, \,\, I^{0-}(\bar x) \subset \tilde I^{0-}_k, \,\,
	  I^{+0}(\bar x) \subset \tilde I^{+0}_k, \,\, I^{+-}(\bar x) \subset \tilde I^{+-}_k.
	\end{equation}
	
	However, feasibility of $\bar x$ for \eqref{eq : genproblem} yields
	\[V = I^{0+}(\bar x) \cup I^{00}(\bar x) \cup I^{0-}(\bar x) \cup I^{+0}(\bar x) \cup I^{+-}(\bar x)\]
	and the index sets $\tilde I^{0+}_k, \tilde I^{00}_k, \tilde I^{0-}_k, \tilde I^{+0}_k, \tilde I^{+-}_k$ are
	pairwise disjoint subsets of $V$ by definition. Hence we claim that \eqref{eq : IndexSetSub} must in fact hold with equalities.
	Indeed, e.g.
	\[\tilde I^{0+}_k \subset V \setminus (\tilde I^{00}_k \cup \tilde I^{0-}_k \cup \tilde I^{+0}_k \cup \tilde I^{+-}_k)
	\subset V \setminus (I^{00}(\bar x) \cup I^{0-}(\bar x) \cup I^{+0}(\bar x) \cup I^{+-}(\bar x)) = I^{0+}(\bar x).\]
	This finishes the proof.
      \end{proof}

      Note that if we assume that there exist a constant $L > 0$, a number $N \in \N$ and a limit point $\bar x$ such that for all $k \geq N$ it holds that
      \[\norm{x_{k+1} - \bar x}_{\infty} \leq L \norm{x_{k+1} - x_k}_{\infty},\]
      by setting $\epsilon_k := \sqrt{\norm{x_{k} - x_{k-1}}_{\infty}}$ we obtain \eqref{eq : EpsK}, since
      \[\frac{\sqrt{\norm{x_{k} - x_{k-1}}_{\infty}}}{\norm{x_{k} - \bar x}_{\infty}} \geq
      \frac{\sqrt{\norm{x_{k}- \bar x}_{\infty}}}{\sqrt{L} \norm{x_{k}- \bar x}_{\infty}} =
      \frac{1}{\sqrt{L \norm{x_{k} - \bar x}_{\infty}}} \to \infty.\]

\section{Numerical results}
  
  Algorithm \ref{AlgMPCC} was implemented in MATLAB.
  To perform numerical tests we used a subset of test problems considered in the thesis of Hoheisel \cite{Ho09}.
  
  First we considered the so-called academic example
  \begin{equation} \label{eq : academic}
    \begin{array}{rl}
	\min\limits_{x \in \mathbb{R}^{2}} & 4x_1 + 2x_2 \\
	\textrm{subject to } & x_1 \geq 0, \\
	& x_2 \geq 0, \\
	& (5 \sqrt{2} - x_1 - x_2)x_1 \leq 0, \\
	& (5 - x_1 - x_2)x_2 \leq 0. \\
      \end{array}
  \end{equation}
  As in \cite{Ho09}, we tested 289 different starting points $x^0$ with $x^0_1,x^0_2 \in \{-5,-4,\ldots,10,20 \}$.
  For 84 starting points our algorithm found a global minimizer $(0,0)$ with objective value 0, while for the
  remaining 205 starting points a local minimizer $(0,5)$ with objective value 10 was found. Hence, convergence to
  the perfidious candidate $(0,5 \sqrt{2})$, which is not a local minimizer, did not occur (see \cite{Ho09}).
  
  Expectantly, after adding constraint $3 - x_1 - x_2 \leq 0$ to the model \eqref{eq : academic}, to artificially exclude the point $(0,0)$,
  unsuitable for the practical application, we reached the point $(0,5)$, now a global minimizer.
  For more detailed information about the problem we refer the reader to \cite{Ho09} and \cite{AchHoKa13}.
  
  Next we solved 2 examples in truss topology optimization, the so called Ten-bar Truss and Cantilever Arm.
  The underlying model for both of them is as follows:
  \begin{equation} \label{eq : trusstopology}
    \begin{array}{rll}
	\min\limits_{(a,u) \in \mathbb{R}^{N} \times \mathbb{R}^{d}} &  V := \sum_{i=1}^N \ell_i a_i \\
	\textrm{subject to } & K(a)u = f,& \\
	& f u \leq c, & \\
	& a_i \leq \bar a_i & i \in \{1,2,\ldots,N\}, \\
	& a_i \geq 0 & i \in \{1,2,\ldots,N\}, \\
	& (\sigma_i(a,u)^2 - \bar\sigma^2)a_i \leq 0 & i \in \{1,2,\ldots,N\}. \\
      \end{array}
  \end{equation}
  Here the matrix $K(a)$ denotes the global stiffness matrix of the structure $a$ and the vector $f \in \mathbb{R}^d$ contains the
  external forces applying at the nodal points. Further, for each $i$ the function $\sigma_i(a,u)$ denotes the
  stress of the $i-$th potential bar and $c, \bar a_i, \bar\sigma$ are positive constants. Again, for more background
  of the model and the following truss topology optimization problems we refer to \cite{Ho09}.
  
  \begin{figure}
    \includegraphics[width=\textwidth]{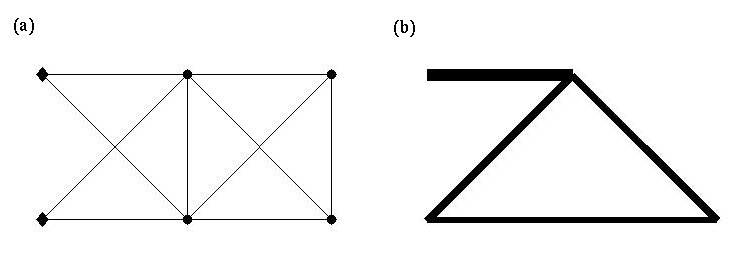}
    \caption{Ten-bar Truss example} \label{fig : TenBar}
  \end{figure}
  
  In the Ten-bar Truss example we consider the ground structure depicted in Figure \ref{fig : TenBar}(a) consisting of $N = 10$ potential bars and
  6 nodal points. We consider a load which applies at the bottom right hand node pulling vertically to the ground with force
  $\norm{f} = 1$. The two left hand nodes are fixed, and hence the structure has $d = 8$ degrees of freedom
  for displacements.
  
  We set $c := 10, \bar a := 100$ and $\bar\sigma := 1$ as in \cite{Ho09} and the resulting structure consisting of 5 bars
  is shown in Figure \ref{fig : TenBar}(b) and is the same as the one in \cite{Ho09}.
  For comparison, in the following table we show the full data containing also the stress values.
  
  \begin{center}
    \begin{tabular}{|ccc|c|}
      \hline
      $i$ & $a_i^*$ & $\sigma_i(a^*,u^*)$ & $u_i^*$ \\
      \hline
      1 & 0 &  1.029700000000000 & -1.000000000000000 \\
      2 & 1.000000000000000 &  1.000000000000000 & 1.000000000000000 \\
      3 & 0 &  1.119550000000000 & -2.000000000000000 \\
      4 & 1.000000000000000 &  1.000000000000000 & 1.302400000000000 \\
      5 & 0 &  0.485150000000000 & -1.970300000000000 \\
      6 & 1.414213562373095 &  1.000000000000000 & -3.000000000000000 \\
      7 & 0 &  0.302400000000000 & -8.000000000000000 \\
      8 & 1.414213562373095 &  1.000000000000000 & -6.511800000000000 \\ \cline{4-4}
      9 & 2.000000000000000 &  1.000000000000000 & $f^T u^* = 8$ \\                  
     10 & 0 &  1.488200000000000 & $V^*= 8.000000000000002$ \\              
      \hline
    \end{tabular}
  \end{center}
  
  We can see that although our final structure and optimal volume are the same as the final structure and the optimal volume in \cite{Ho09},
  the solution $(a^*,u^*)$ is different. For instance, since $f^T u^* = 8 < 10 = c$, our solution does not reach the maximal compliance.
  Similarly as in \cite{Ho09}, we observe the effect of vanishing constraints since the stress values from the table show
  that
  \[\sigma_{max}^* := \max_{1 \leq i \leq N} \vert \sigma_{i}(a^*,u^*) \vert = 1.4882 >
  \hat\sigma^* := \max_{1 \leq i \leq N : a_i^* > 0} \vert \sigma_{i}(a^*,u^*) \vert = 1 = \bar\sigma.\]
  
  \begin{figure}
    \includegraphics[width=\textwidth]{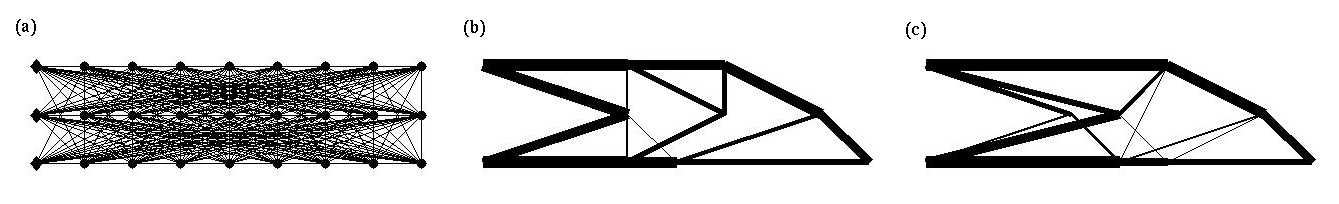}
    \caption{Cantilever Arm example} \label{fig : CantilArm}
  \end{figure}
  
  In the Cantilever Arm example we consider the ground structure depicted in Figure \ref{fig : CantilArm}(a) consisting of $N = 224$ potential bars and
  27 nodal points. Again, we consider a load acting at the bottom right hand node pulling vertically to the ground with force
  $\norm{f} = 1$. Now the three left hand nodes are fixed, and hence $d = 48$.
  
  We proceed as in \cite{Ho09} and we first set $c := 100, \bar a := 1$ and $\bar\sigma := 100$. The resulting structure consisting of only 24 bars
  (compared to 38 bars in \cite{Ho09}) is shown in Figure \ref{fig : CantilArm}(b).
  Similarly as in \cite{Ho09}, we have $\max_{1 \leq i \leq N} a_i^{*1} = \bar a$ and $f u^{*1} = c$.
  On the other hand, our optimal volume $V^{*1} = 23.4407$ is a bit larger than the optimal volume 23.1399 in \cite{Ho09}.
  Also, analysis of our stress values shows that
  \[\sigma_{max}^{*1} := \max_{1 \leq i \leq N} \vert \sigma_{i}(a^{*1},u^{*1}) \vert = 60.4294 >>
  \hat\sigma^{*1} := \max_{1 \leq i \leq N : a_i^{*1} > 0} \vert \sigma_{i}(a^{*1},u^{*1}) \vert = 2.6000\]
  and hence, although it holds true that both absolute stresses as well as absolute ''fictitious stresses'' (i.e., for zero bars)
  are small compared to $\bar\sigma$ as in \cite{Ho09}, the difference is that in our case they are not the same.
  
  The situation becomes more interesting when we change the stress bound to $\bar\sigma= 2.2$.
  The obtained structure consisting again of only 25 bars (compared to 37 or 31 bars in \cite{Ho09}) is shown in Figure \ref{fig : CantilArm}(c).
  As before we have $\max_{1 \leq i \leq N} a_i^{*2} = \bar a$ and $f u^{*2} = c$.
  Our optimal volume $V^{*2} = 23.6982$ is now much closer to the optimal volumes 23.6608 and 23.6633 in \cite{Ho09}.
  Similarly as in \cite{Ho09}, we clearly observe the effect of vanishing constraints since our stress values show
  \[\sigma_{max}^{*2} := \max_{1 \leq i \leq N} \vert \sigma_{i}(a^{*2},u^{*2}) \vert = 24.1669 >>
  \hat\sigma^{*2} := \max_{1 \leq i \leq N : a_i^{*2} > 0} \vert \sigma_{i}(a^{*2},u^{*2}) \vert = 2.2 = \bar\sigma. \]
  Finally, we obtained 32 bars (in contrast to 24 bars in \cite{Ho09}) satisfying both
  \[a_i^{*2} < 0.005 = 0.005 \bar a \, \, \textrm{ and } \, \, \vert \sigma_{i}(a^{*2},u^{*2}) \vert > 2.2 = \bar\sigma. \]
  
  To better demonstrate the performance of our algorithm we conclude this section by a table with more detailed information
  about solving Ten-bar Truss problem and 2 Cantilever Arm problems (CA1 with $\bar\sigma := 100$ and CA2 with $\bar\sigma := 2.2$).
  We use the following notation.
  \begin{center}
    \begin{tabular}{|c|c|c|c|c|c|}
      \hline
Problem&name of the test problem \\
$(n,q)$&number of variables, number of all constraints \\
$k^*$&total number of outer iterations of the SQP method \\
$(N_0, \ldots, N_{k^*-1})$&total numbers of inner iterations corresponding to each outer iteration \\
$\sum_{k=0}^{k^*-1}j(k)$&overall sum of steps made during line search\\
$\sharp f_{eval}$&total number of function evaluations, $\sharp f_{eval} = k^* + \sum_{k=0}^{k^*-1}j(k)$\\
$\sharp \nabla f_{eval}$&total number of gradient evaluations, $\sharp \nabla f_{eval} = k^*+1$\\
      \hline
    \end{tabular}
  \end{center}

  \begin{center}\small
    \begin{tabular}{|c|c|c|c|c|c|c|}
      \hline
      Problem & $(n,q)$ & $k^*$ & $(N_0, \ldots, N_{k^*-1})$ & $\sum_{k=0}^{k^*-1}j(k)$ & $\sharp f_{eval}$ & $\sharp \nabla f_{eval}$\\
      \hline
      Ten-bar Truss & $(18,39)$ & $14$ & $(1, \ldots, 1, 2, 2, 2, 2, 1, 1)$ & $67$ & $81$ & $15$ \\
      CA1 & $(272,721)$ & $401$ & $(1, \ldots, 1)$ & $401$ & $802$ & $402$ \\
      CA2 & $(272,721)$ & $1850$ & $(1, \ldots, 1)$ & $1850$ & $3700$ & $1851$ \\
      \hline
    \end{tabular}
  \end{center}
  
\section*{Acknowledgments} This work was supported by the Austrian Science Fund (FWF) under grant P 26132-N25.

\end{document}